\documentclass[twoside]{amsart}
\usepackage{latexsym, amsmath,amssymb,  bbold}

\usepackage{subcaption}

	\usepackage{xcolor}
	
	\usepackage{graphicx}

\renewcommand{\epsilon}{\varepsilon}
\newcommand{\newsection}[1] 
{\subsection{#1}\setcounter{theorem}{0} \setcounter{equation}{0} 
\par\noindent}

\newtheorem{theorem}{Theorem}

\newtheorem{lemma}[theorem]{Lemma}
\newtheorem{corr}[theorem]{Corollary}

\newtheorem{proposition}[theorem]{Proposition}
\newtheorem{deff}[theorem]{Definition}

\newcommand{\bth}{\begin{theorem}}
\newcommand{\ble}{\begin{lemma}}
\newcommand{\bcor}{\begin{corr}}

\newcommand{\bdeff}{\begin{deff}}

\newcommand{\bprop}{\begin{proposition}}
\newcommand{\ele}{\end{lemma}}
\newcommand{\ecor}{\end{corr}}
\newcommand{\edeff}{\end{deff}}

\newcommand{\eprop}{\end{proposition}}

\newcommand{\la}{\lambda}

\newcommand{\e}{\varepsilon}

\renewcommand{\Pi}{\varPi}

\renewcommand{\epsilon}{\varepsilon}

\newcommand{\R}{{\mathbb R}}

\newcommand{\1}{\mathbb{1}}

\newcommand{\diag}{\Upsilon^{\text{diag}}}

\newcommand{\tdiag}{\tilde\Upsilon^{\text{diag}}}

\newcommand{\far}{\Upsilon^{\text{far}}}

\newcommand{\Atn}{A^{\theta_0}_\nu}

\newcommand{\tAtn}{\tilde A^{\theta_0}_{\nu}}

\newcommand{\Atnt}{A^{\theta_0}_{\tilde\nu}}

\newcommand{\tAtnt}{\tilde A^{\theta_0}_{\tilde \nu}}

\newcommand{\xid}{\Xi_{\theta_0}}

\begin{document}

\title[Strichartz estimates  on nonpostively curved curved manifolds]
{Strichartz estimates for the Schr\"odinger equation on  compact manifolds with  nonpositive sectional curvature}

\thanks{The first author was supported in part by an AMS-Simons travel grant. The second author was supported in part by the NSF (DMS-1665373 and DMS-2348996)
and a Simons Fellowship. This research was partly carried out while the first
author was at the University of Maryland.}


\author{Xiaoqi Huang}
\address[X.H.]{Department of Mathematics, Louisiana State University, Baton Rouge, LA 70808}
\email{xhuang49@lsu.edu}
\author{Christopher D. Sogge}
\address[C.D.S.]{Department of Mathematics,  Johns Hopkins University,
Baltimore, MD 21218}
\email{sogge@jhu.edu}

\begin{abstract}
We obtain  improved Strichartz estimates for solutions of the Schr\"odinger equation on  compact manifolds with  nonpositive 
sectional curvatures which are related to the
classical universal results of Burq, G\'erard and Tzvetkov~\cite{bgtmanifold}. More explicitly, we are able refine the arguments in the recent work of Blair and the authors \cite{blair2023strichartz}  to obtain  no-loss $L^p_tL^{q}_{x}$-estimates on intervals of length $\log \la\cdot \la^{-1} $
 for all {\em admissible} pairs $(p,q)$ when the initial data have frequencies
 comparable to $\la$, which, given the role of the Ehrenfest time, is the natural analog in this setting of the universal results
in \cite{bgtmanifold}.   We achieve this log-gain over the universal estimates by applying the Keel-Tao theorem along with improved global kernel estimates for microlocalized operators which
exploit the geometric assumptions.
\end{abstract}

\maketitle

\centerline{ \bf In memoriam: {\em Steve Zelditch (1953-2022)}}

\newsection{Introduction}

The purpose of this paper is to improve  the Strichartz estimates for the 
Schr\"odinger equation on  negatively curved compact manifolds of Blair and the authors \cite{blair2023strichartz}, while at the same time simplifying the arguments there.

Let us first recall the universal estimates
of Burq, G\'erard and Tzvetkov~\cite{bgtmanifold}.
If $(M^d,g)$ is a compact Riemannian manifold of
dimension $d\ge2$, then the main estimate
in \cite{bgtmanifold} is that if $\Delta_g$ is
the associated Laplace-Beltrami operator and
\begin{equation}\label{00.1}
u(x,t)=\bigl(e^{-it\Delta_g}f\bigr)(x)
\end{equation}
is the solution of the Schr\"odinger equation
on $M^d\times \R$,
\begin{equation}\label{00.2}
i \partial_tu(x,t)=\Delta_gu(x,t), \quad u(x,0)=f(x),
\end{equation}
then if we define 
$$\|u\|_{L^p_tL^q_x(M^{d}\times [0,1])}=\bigr(\, \int_0^1 \, \|u(\, \cdot\, , t)\|_{L^q_x(M^{d})}^p \, dt\,\bigr)^{1/p},$$
one has the mixed-norm Strichartz estimates
\begin{equation}\label{00.3}
\|u\|_{L^p_tL^q_x(M^d\times [0,1])}
\lesssim \|f\|_{H^{1/p}(M^d)}
\end{equation}
for all {\em admissible} pairs $(p,q)$.  By the latter
we mean, as in Keel and Tao~\cite{KT},
\begin{equation}\label{00.4}
d(\tfrac12-\tfrac1q)=\tfrac2p\, \, \,
\text{and } \, \, 2< q\le \tfrac{2d}{d-2} \, \,
\text{if } \, d\ge 3, \, \, \, 
\text{or } \, 2< q<\infty \, \, 
\text{if } \, \, d=2.
\end{equation}
Also, in \eqref{00.3} 
$H^\mu$ denotes the
standard Sobolev space
\begin{equation}\label{00.5}
\|f\|_{H^\mu(M^d)}=
\bigl\| \, (I+P)^\mu f\, \bigr\|_{L^2(M^d)}, 
\quad \text{with } \, \, P=\sqrt{-\Delta_g},
\end{equation}
and ``$\lesssim$'' in \eqref{00.3} and, in what
follows, denotes an inequality with an implicit,
but unstated, constant $C$ which can change at each occurrence.

Note that if $e_\la$ is an eigenfunction of
$P$ with eigenvalue $\la$, i.e.,
\begin{equation}\label{00.6}
-\Delta_g e_\la = \la^2 e_\la,
\end{equation}
then
\begin{equation}\label{00.7}
u(x,t)=e^{it\la^2}e_\la(x)
\end{equation}
solves \eqref{00.2} with initial data $f=e_\la$. Thus, 
unlike  Euclidean space, or  hyperbolic space, one cannot replace $[0,1]$ by
$\R$ in \eqref{00.3} due to the existence of eigenfunctions.
Also, for the endpoint Strichartz estimates where
$p=2$ and $q=\tfrac{2d}{d-2}$ with $d\ge3$, the derivative loss of $\frac12$ in \eqref{00.3} can not be improved on the standard round sphere $S^d$, by taking the initial data $f$ to be zonal eigenfunctions.  See, e.g.,  \cite{bgtmanifold} and \cite{blair2023strichartz} for more details. 

However, on compact manifolds with other types of geometries, \eqref{00.3} can be improved, see e.g., \cite{BoDe}, \cite{DGG} and \cite{DGGM} for the torus case and  \cite{blair2023strichartz} for general compact manifolds with non-positive sectional curvatures. Also, for {\em admissible} pairs $(p,q)$ other than $(2, \frac{2d}{d-2})$, one can have better estimates than \eqref{00.3} on the sphere as well using the specific
arithmetic properties of the distinct eigenvalues of the Laplacian on $S^d$.  See, e.g., Theorem 4 in \cite{bgtmanifold}.

As in \cite{blair2023strichartz}, to align with the numerology in related earlier results
involving eigenfunction and spectral projection estimates,  in what follows,
we shall always take $d=n-1$. 
Our main result which improves on estimates in \cite{blair2023strichartz} then is the following

\begin{theorem}\label{nonposthm}  Let $M^{n-1}$ be a
$d=n-1\ge 2$ dimensional compact manifold all of whose
sectional curvatures are nonpositive.  Then for all {\em admissible} pairs $(p,q)$,
\begin{equation}\label{00.10}
\|u\|_{L^p_tL^q_x(M^{n-1}\times [0,1])}\lesssim \bigl\|(I+P)^{1/p} \, (\log(2I+P))^{-\frac{1}{p}}f \bigr\|_{L^2(M^{n-1})}.
\end{equation}
\end{theorem}
If we let the initial data $f=e_\la$, then by \eqref{00.7} and the special case $(p,q)=(2,\frac{2d}{d-2})$ of  \eqref{00.10}, we have  
\begin{equation}\label{eigengain}
    \|e_\la\|_{L^{\frac{2d}{d-2}}(M^d)}
\lesssim \la^{1/2}\, (\log\la)^{-1/2}\|e_\la\|_{L^2
(M^d)},
\end{equation}
which gives a $(\log\la)^{-1/2}$ gain compared with the universal eigenfunction estimates of the second author \cite{sogge88}. 
For manifolds with nonpositive curvature, \eqref{eigengain} was first proved by 
 Hassell and Tacy \cite{HassellTacy}, with similar results for all $q>\tfrac{2(d+1)}{d-1}$. Thus \eqref{00.10} is a natural generalization of \eqref{eigengain} for solutions of Schr\"odinger equation and also provides a novel approach to  get improved eigenfunction estimates.

By the Littlewood-Paley theory,
we may reduce \eqref{00.10} to proving certain dyadic estimates. More explicitly,  let us fix a Littlewood-Paley bump function
$\beta$ satisfying
\begin{equation}\label{00.11}
\beta \in C^\infty_0((1/2,2)) \quad
\text{and } \, \, 1=\sum_{k=-\infty}^\infty 
\beta(2^{-k}s), \, \, s>0.
\end{equation}
Then, if we set $\beta_0(s)=1-\sum_{k=1}^\infty
\beta(2^{-k}s)\in C^\infty_0(\R_+)$ and
$\beta_k(s)=\beta(2^{-k}s)$, $k=1,2,\dots$, we
have (see, e.g., \cite{SFIO2})
\begin{equation}\label{00.12}
\|h\|_{L^q(M^{n-1})}\approx
\bigl\| \, (\, \sum_{k=0}^\infty |\beta_k(P)h|^2\, 
)^{1/2} \, \bigr\|_{L^q(M^{n-1})}, \, \, \,
1<q<\infty.
\end{equation} 
Trivially, $\|\beta_0(P)e^{-it\Delta_g}\|_{L^2(M^{n-1})
\to {L^p_tL^q_x(M^{n-1}\times [0,1])}}=O(1)$, and, similarly such results
where $k=0$ is replaced by a small fixed $k\in {\mathbb N}$ are
also standard.
So, as noted
in \cite{bgtmanifold}, one can use \eqref{00.12} and
Minkowski's inequality to see that
\eqref{00.3} follows
from the uniform bounds
\begin{equation}\label{00.3'}
\| e^{-it\Delta_g}\beta(P/\la)f\|_{L^p_tL^q_x(M^{n-1}\times [0,1])}\le C\la^{\frac1{p}}\, \|f\|_{L^2(M^{n-1})}, \quad \la \gg 1.
\end{equation}
Burq, G\'erard and Tzvetkov proved this estimate in
\cite{bgtmanifold} by showing that one always has the 
following uniform dyadic estimates over very small intervals:
\begin{equation}\label{00.3''}
\|e^{-it\Delta_g} \beta(P/\la)f\|_{L^p_tL^q_x(M^{n-1}\times [0,\la^{-1}])}\le C\, \|f\|_{L^2(M^{n-1})}, \quad \la \gg 1.
\end{equation}
It is not hard to see that 
 \eqref{00.3''}  yields 
\eqref{00.3'}, since one can write $[0,1]$ as the
union of $\approx \la$ intervals of length $\la^{-1}$
and thus obtain \eqref{00.3'} by adding up the
uniform estimates on each of these subintervals
that \eqref{00.3''} affords. Also, the  bounds in \eqref{00.3''} cannot be improved
on {\em any} manifold by taking $f(x)=f_\la(x)=\beta(P/\la)(x,x_0)$ with $\beta(P/\la)(x,y)$ being the kernel of the 
Littlewood-Paley operators. As a result, to obtain
improvements such as those in \eqref{00.10},
larger time intervals must be used.   Specifically, we shall show that if
$M^{n-1}$ is as in Theorem~\ref{nonposthm} then
we have the uniform bounds
\begin{equation}\label{00.10'}
\| e^{-it\Delta_g} \beta(P/\la) f\|_{L^p_tL^q_x(M^{n-1}
\times [0, \, \log\la\cdot \la^{-1}])}\le C \, \|f\|_{L^2(M^{n-1})}, \quad \la\gg 1,
\end{equation}
which is a  natural extension of the uniform
small-time scale estimates \eqref{00.3''}, and  perhaps the largest one can hope to obtain
such estimates in the geometry we are focusing on  using available techniques, due to
the role of the Ehrenfest time.
Also, by the above counting
arguments, one obtains \eqref{00.10} from
\eqref{00.10'} since $[0,1]$ can be covered
by $\approx \la/\log\la$ intervals of length
$\log\la\cdot \la^{-1}$.

The bound in \eqref{00.10'} improves the  result in \cite{blair2023strichartz} in two aspects, first, it removes
the power of $\log\la$ loss there, and second, it includes all  {\em admissible} pairs $(p,q)$. The main ideas in the proof of \eqref{00.10'} are similar to those in \cite{blair2023strichartz}, both of which involve a height decomposition. The larger values can be dealt with using kernel estimates for certain {\em global} operators, while the smaller values involves the use of a  Whitney type decomposition and  bilinear oscillatory integral
estimates of Lee \cite{LeeBilinear}. Compared to \cite{blair2023strichartz},  we are able to improve 
 the arguments used for the {\em diagonal} term in the  Whitney decomposition, which previously relied on a microlocalized version of improved Strichartz estimates. We avoid this by applying the abstract theorem of Keel-Tao \cite{KT} to some 
new Banach space with $L^p$ norm depending on the $\ell^p$ norm of different microlocalized pieces. The analogous dispersive estimate adapted to the new space, which is necessary to apply the Keel-Tao theorem, is obtained by using the kernel estimates involving the microlocalized operators as proved in \cite{blair2023strichartz}.

If  $\chi_{[\la,\la+(\log\la)^{-1}]}$ denotes the spectral projection operator for $P=\sqrt{-\Delta_g}$ associated with the interval $[\la,\la+(\log\la)^{-1}]$, then a simple consequence
of \eqref{00.10'} is that for $q_e=\tfrac{2d}{d-2}$, $d=n-1\ge 3$,
\begin{equation}\label{bern}
\| \chi_{[\la,\la+(\log\la)^{-1}]}f\|_{L^{q_e}(M^d)} \le C\la (\log\la)^{-1/2}\|f\|_{L^2(M^d)},
\end{equation}
assuming that $M^d$ has nonpositive sectional curvatures.
This in turn, implies the natural sharp spectral projection estimates for these operators for all $q\ge q_e$ under these curvature assumptions, which are due to B\'erard~\cite{Berard} for
$q=\infty$  and later generalized by Hassell and Tacy~\cite{HassellTacy} to $q>\tfrac{2(d+1)}{d-1}$ and $d\ge2$.

To prove \eqref{bern} fix $a\in {\mathcal S}({\mathbb R})$ satisfying $a(0)=1$ and
$\Hat a(\tau)=0$ if $|\tau|\ge1$.  Then for $q_e$ and $d=n-1$ as in \eqref{bern}, by Bernstein's inequality and \eqref{00.10'},
\begin{align*}
\| &\chi_{[\la,\la+(\log\la)^{-1}]}f\|_{L^{q_e}(M^{n-1})}
\\
&\le \| a(\la(\log\la)^{-1}t) \, e^{-it\Delta_g}\beta(P/\la) \chi_{[\la,\la+(\log\la)^{-1}]}f\|_{L^\infty_t L^{q_e}_x(M^d\times \R)}
\\
&\le C\la^{\frac12}(\log\la)^{-\frac12}
\| a(\la(\log\la)^{-1}t) \, e^{-it\Delta_g}\beta(P/\la) \chi_{[\la,\la+(\log\la)^{-1}]}f\|_{L^2_t L^{q_e}_x(M^d\times \R)}
\\
&\le C\la^{\frac12}(\log\la)^{-\frac12} \|f\|_{L^2(M^{n-1})}.
\end{align*}
Note that this argument implies that any improvements of \eqref{00.10'} to include uniform bounds for intervals of size
$\delta(\la)$ with $\delta(\la)/[\la^{-1}\cdot \log\la ]\nearrow \infty$ would imply sharp spectral projection estimates associated
with spectral windows of length $o((\log\la)^{-1})$ for $q\ge q_e$ under the assumption of nonpositive curvature, which, even for the case
of $q=\infty$, seems difficult.  There has been no such improvement of the sup-norm estimates of B\'erard~\cite{Berard} in the last five decades,
and progress of this nature has been elusive due to the role of the Ehrenfest time.


In the case of flat tori, if we fix $q=\frac{2(n+1)}{n-1}$, recall the near optimal results of
Bourgain and Demeter~\cite{BoDe} for $n\ge 4$ and Bourgain~\cite{bourgain1993fourier} for $n=2, 3$
\begin{equation}\label{00.10''}
\| e^{-it\Delta_g} \beta(P/\la) f\|_{L^{q}_{t,x}(M^{n-1}
\times [0, \,1])}\le C_\e\la^\e \, \|f\|_{L^2(M^{n-1})}, \, \la\gg 1, \,\,\forall\e>0.
\end{equation}
The proof of \eqref{00.10''} is based on number-theoretic methods for $n=2,3$, while for $n\ge 4$, the estimates were derived via $\ell^2$-decoupling methods. More recently, for $n=3$, Herr and Kwak \cite{herr2023strichartz} obtained a
lossless version of \eqref{00.10''} for any time intervals of length $(\log\la)^{-1}$ by using a new method based on incidence geometry, and the length of the interval can not be extended further by testing against $f=\sum_{k\in [-\la, \la]^2\cap \mathbb{Z}^2}e^{ik\cdot x}$. See e.g, \cite{bourgain1993fourier} for more details.

By using the  favorable properties of the universal cover $\R^n$ and the fact that the types of microlocal cutoffs  we
shall employ commute well with  Schr\"odinger
propagators,
it seems likely that we shall be able to  modify the arguments in the proof of Theorem 1.1 to obtain no loss  dyadic estimates on tori
${\mathbb T}^n$
 on intervals
of length $\la^{-1+\delta_n}$ for some $0<\delta_n<1$, which would be a natural generalization of the results in \cite{herr2023strichartz} for all  {\em admissible} pairs $(p,q)$ in any dimension $n\ge 3$. We hope to explore this problem as well as possible improved Strichartz estimates for
spheres in a later work.


This paper is organized as follows.  In the next section
we present the main arguments that allow us to prove
 Theorem~\ref{nonposthm}. The proof requires local bilinear arguments
from harmonic analysis, which are discussed in Section 3.

\newsection{Main arguments}


To start, let $\beta$ be the Littlewood-Paley bump function in \eqref{00.11}, and also fix
\begin{equation}\label{22.1}
\eta\in C^\infty_0((-1,1)) \quad
\text{with } \, \, \eta(t)=1, \, \, \, |t| \le 1/2.
\end{equation}
We then shall consider the dyadic time-localized dilated Schr\"odinger operators
\begin{equation}\label{22.2} 
S_\la =\eta(t/T) e^{-it\la^{-1}\Delta_g} \beta(P/\la),
\end{equation}
and claim that the estimates in Theorems \ref{nonposthm} is a consequence of the following.

\begin{proposition}\label{mainprop}  Let $M^d$, $d=n-1\ge2$ be a fixed compact manifold all of whose
sectional curvatures are nonpositive, and $(p,q)$ be as in \eqref{00.4}.  Then we can fix $c_0>0$ so that for large $\la \gg 1$ we have the
uniform bounds
\begin{equation}\label{22.3}
\| S_\la f\|_{L^p_tL^q_x(M^{n-1}\times [0,T])}\le C\la^{\frac1{p}}  \, \|f\|_{L^2(M^{n-1})}, \quad
\text{if } \, \, T=c_0\log\la.
\end{equation}
\end{proposition}

We claim that \eqref{22.3}  implies Theorem \ref{nonposthm}.  First note that, by changing scales, \eqref{22.1} and \eqref{22.3} imply that for large
enough $\la$ we have the analog of \eqref{00.10'} where the interval
$[0,\log\la\cdot \la^{-1}]$ in the left is replaced by
$[0, \tfrac12c_0\log\la\cdot \la^{-1}]$, and this of course implies \eqref{00.10'}
at the expense of including an additional factor of $(c_0/2)^{-1/q_c}$ in the constant in 
the right if $c_0<2$.  As we indicated before, the estimate \eqref{00.10'} for
large $\la$ and Littlewood-Paley theory yield Theorem~\ref{nonposthm}, which verifies
our claim regarding \eqref{22.3}.  

Also note that if we replace $[0, T]$ by $[0,1]$, then by \eqref{00.3''} and a rescaling argument we have 
\begin{equation}\label{22.3''}
\| S_\la f\|_{L^p_tL^q_x(M^{n-1}\times [0,1])}\le C\la^{\frac1{p}}  \, \|f\|_{L^2(M^{n-1})},
\end{equation}
which hold on a general smooth compact manifold.

To prove \eqref{22.3}, note  that since the case $(p,q)=(\infty, 2)$ is trivial, by interpolation, it suffices to consider the $(p,q)$ pairs which satisfy
\begin{equation}\label{00.4'}
(p, q)=(2,\tfrac{2d}{d-2}) \, \,
\text{if } \, d\ge 3, \, \, \, 
\text{or } \, d(\tfrac12-\tfrac1q)=\tfrac2p, \,\,4< q<\infty \, \, 
\text{if } \, \, d=2.
\end{equation}
The condition $q>4$ is equivalent to $q>p$ when $d=2$,
this will allow us to simplify some of the calculations to follow. 

As in \cite{blair2023strichartz}, we need to introduce a few auxiliary operators that allow us to use bilinear
techniques. First of all, we need  to compose the ``global operators'' $S_\la$ with related
local ones.  Motivated  by the recent work of the  authors \cite{HSSchro},
our ``local'' auxiliary operators will be  the following ``quasimode'' operators adapted to the
scaled Schr\"odinger operators $\la D_t+\Delta_g$, 
\begin{equation}\label{22.5}
\sigma_\la = \sigma\bigl(\la^{1/2}|D_t|^{1/2}-P\bigr) \, \tilde \beta(D_t/\la),
\end{equation}
where
\begin{equation}\label{22.6}
\sigma\in {\mathcal S}(\R) \, \, \text{satisfies } \, \,
\sigma(0)=1 \, \, \text{and } \, \, \text{supp }\Hat \sigma
\subset \delta\cdot [1-\delta_0,1+\delta_0]=[\delta-\delta_0\delta, \delta+\delta_0\delta],
\end{equation}
with $0<\delta,\delta_0<1/8$ to be specified later, and, also here
\begin{equation}\label{22.7}
\tilde \beta\in C^\infty_0((1/8,8)) \quad \text{satisfies } \, \,
\tilde \beta=1 \, \, 
\text{on } \, \, [1/6,6].
\end{equation}
Here the properties of $\sigma$, as well as the small constants $\delta$ and $\delta_0$ are the same as those in \cite{blair2023strichartz}, which allows us to use the bilinear oscillatory integral estimates in \cite{blair2023strichartz} in the next section. The smallness of $\delta$ is also related to another initial 
  microlocalization that is needed for the bilinear arguments, as we shall describe below.
  
Let us write
\begin{equation}\label{22.8}
I=\sum_{j=1}^N B_j(x,D),
\end{equation}
where each $B_j\in S^0_{1,0}(M^{n-1})$ is a standard
pseudo-differential operator with symbol supported in
a small conic neighborhood of some $(x_j,\xi_j)\in
S^*M$.  The size of the support will be described
later; however, these operators will not depend
on our parameter $\la\gg 1$.  Next, if 
$\tilde \beta$ is as in \eqref{22.7} then the
dyadic operators
\begin{equation}\label{22.9}
B=B_{j,\la} = B_j\circ \tilde \beta(P/\la)
\end{equation}
are uniformly bounded on $L^p$, i.e.,
\begin{equation}\label{22.10}
\|B\|_{L^p(M^{n-1})\to L^p(M^{n-1})}
=O(1) \quad \text{for } \, \, 1\le p\le \infty.
\end{equation}
Also, note that since $\sigma\in {\mathcal S}(\R)$
a simple calculation shows that if $\la_k$ is an
eigenvalue of $P$
$$(1-\tilde \beta(\la_k/\la))
\, \sigma(\la^{1/2}|\tau|^{1/2}-\lambda_k) \, 
\tilde \beta(\tau/\la)
=O(\la^{-N}(1+\la_k+|\tau|)^{-N}) \, \, \forall \, N.
$$
Consequently,
$$\|\sigma_\la -\tilde \beta(P/\la)\circ \sigma_\la
\|_{L^2(M^{n-1}\times [0,T])\to {L^p_tL^q_x(M^{n-1}\times [0,T])}}
=O(\la^{-N}) \quad \forall \, N.$$

Thus, if $B_j$ is as in \eqref{22.8} and 
$B_{j,\la}$ is the corresponding dyadic operator
in \eqref{22.9}
\begin{equation}\label{22.11}
\|B_j\sigma_\la -B_{j,\la}\sigma_\la\|_{L^2(M^{n-1}\times [0,T])\to {L^p_tL^q_x(M^{n-1}\times [0,T])}}
=O(\la^{-N}) \quad \forall \, N,
\end{equation}
since operators in $S^0_{1,0}(M^{n-1})$ are bounded on $L^p$ for $1<p<\infty$.

We need one more result for now about these local
operators:

\begin{lemma}\label{lemmadiff}
If $S_\la$ as in \eqref{22.2}, $\sigma_\la$ is as in \eqref{22.5} and $(p,q)$ are as in \eqref{00.4}, then
\begin{equation}\label{22.12}
\|(I-\sigma_\la)\circ S_\la f\|_{L^p_tL^q_x(M^{n-1}\times [0,T])}\le CT^{\frac1{p}-\frac12}\la^{\frac1{p}}\|f\|_2.
\end{equation}
We also have 
\begin{equation}\label{qu1}
\bigl\|  \sigma_\la F\bigr\|_{L^p_tL^q_x}
\le C\la^{\frac1{p}}\|F\|_{L^2_{t,x}}.
\end{equation}
\end{lemma}
Lemma~\ref{lemmadiff} is a simple generalization of the results in  Lemma 2.2 and 2.3 of \cite{blair2023strichartz} to all {\em admissible} pairs $(p,q)$. Using the the local dyadic Strichartz estimates \eqref{00.3''}, 
the proof of Lemma~\ref{lemmadiff} follows from the same arguments as in \cite{blair2023strichartz},  so we skip the details here.

For a given $B=B_{j,\la}$ as in \eqref{22.9} let
us define the microlocalized variant of 
$\sigma_\la$ as follows
\begin{equation}\label{22.13}
\tilde \sigma_\la = B\circ \sigma_\la, 
\quad B=B_{j,\la},
\end{equation}
and the associated ``semi-global'' operators
\begin{equation}\label{22.14}
\tilde S_\la =\tilde \sigma_\la \circ S_\la.
\end{equation}

By \eqref{22.8}, \eqref{22.11} and \eqref{22.12},
in order to prove Proposition~\ref{mainprop}, it suffices
to show that if $T=c_0\log\la$ with 
$c_0>0$ sufficiently small (depending on $M^{n-1}$), then, if all the sectional curvatures of
$M^{n-1}$ are nonpositive, for $(p,q)$ be as in \eqref{00.4'}, we have
\begin{equation}\label{22.3'}
\|\tilde S_\la f\|_{L^p_tL^q_x(M^{n-1}\times [0,T])}\le C\la^{\frac1{p}}\|f\|_2.
\end{equation}

As we shall see, in order to prove \eqref{22.3'} we shall need to take $\delta$ and $\delta_0$ in
\eqref{22.6} and \eqref{22.7} to be sufficiently small for each $j$; however, since, by the compactness of $M^{n-1}$
and the arguments to follow, the sum in \eqref{22.8} can be taken to be finite,
 we can take these two parameters to be the minimum over what is needed for $j=1,\dots,N$.

\bigskip

\noindent{\bf 2.1. Height Decomposition.}

Next we set up a variation of an argument of
Bourgain~\cite{BourgainBesicovitch} originally used
to study Fourier transform restriction problems, and, more
recently, to study eigenfunction problems in
\cite{BHSsp}, \cite{SBLog} and \cite{sogge2015improved}.
This involves splitting the estimates in 
Proposition~\ref{mainprop} into two heights
involving relatively large and small values of
$|\tilde S_\la f(x,t)|$.  

To describe this, here, and in what follows we shall
assume, as we just did, that $f$ is $L^2$-normalized, that is
\begin{equation}\label{normalize}\|f\|_2=1.
\end{equation}  Then, we shall prove
the estimates in Proposition~\ref{mainprop}, using
very different techniques by estimating $L^p_tL^q_x$ bounds
over the two regions
\begin{multline}\label{22.24}A_+=\{(x,t)\in  M^{n-1}\times [0,T]: \, |\tilde S_\la f(t,x)|\ge \la^{\frac{n-1}4+\e_1} \},
 \\
  \text{and } \, \, A_-=\{(x,t)\in M^{n-1} \times  [0,T] : \, |\tilde S_\la f(x,t)|< \la^{\frac{n-1}4+\e_1} \}.
  \end{multline}
Due to the numerology of the powers of $\la$ arising,
the splitting occurs at height $\la^{\frac{n-1}4+\e_1}$, where $\e_1>0$ is a small constant that may depend on the dimension $d=n-1$. As we shall see later in \eqref{far} and \eqref{iii1}, we can take $\e_1=\frac 1{100}$ for $n-1\ge 3$ while  for $n-1=2$, the choice of $\e_1$ depends on the exponent $q$ for {\em {admissible}} pairs $(p,q)$, with $\e_1\rightarrow 0$ as $q\rightarrow \infty$.
The transition occurring at, basically, $\la^{\frac{n-1}4}$ is natural and arises
due to Knapp-type phenomena, both in Euclidean problems,
as well as the geometric ones that we are considering here.  

Thus, to prove Proposition~\ref{mainprop}, it suffices to prove the analog \eqref{22.3'} for the two regions in \eqref{22.24}.



\medskip

 \noindent{\bf 2.2. Estimates for relatively large values: Proof of \eqref{22.3'}  on the set $A_+$ .}

We first note that, by Lemma~\ref{lemmadiff} 
and \eqref{22.10}
we have
$$ \|\tilde S_\la f\|_{L^p_tL^q_x(A_+)} \le \| B S_\la f\|_{L^p_tL^q_x(A_+)} + CT^{\frac1{p}-\frac12}\la^{\frac1{p}},
$$
and, since $p\ge2$ for $(p,q)$ as in \eqref{00.4'},  \eqref{22.3'} would follow from
\begin{equation}\label{high'}
\| BS_\la f\|_{L^p_tL^q_x(A_+)}\le C\la^{\frac1{p}}+\tfrac12 \|\tilde S_\la f\|_{L^p_tL^q_x(A_+)}.
\end{equation}

To prove this we shall adapt an argument of
Bourgain~\cite{BourgainBesicovitch} and more
recent variants in
\cite{BHSsp} and  \cite{sogge2015improved} .
Specifically, choose $g(x,t)$ such that
$$\|g\|_{L^{p'}_tL^{q'}_x(A_+)}=1
\quad \text{and } \, \,
\|BS_\la f\|_{L^p_tL^q_x(A_+)}
=\iint BS_\la f\cdot
\overline{\bigl(\1_{A_+}\cdot g\bigr)} \, dx dt.
$$
Then, since we are assuming that $\|f\|_2=1$, by
the Schwarz inequality
\begin{align}\label{22.29}
\|BS_\la f\|^2_{L^p_tL^q_x(A_+)}
&= \Bigl( \, \int f(x) \, \cdot \,  
\overline{\bigl(S^*B^*\bigr)(\1_{A_+}
\cdot g\bigr)(x)} \, dx \, \Bigr)^2
\\
&\le \int |S^*_\la B^*(\1_{A_+}\cdot g)(x)|^2 \, dx
\notag
\\
&=\iint \bigl(BS_\la S^*_\la B^*\bigr)(\1_{A_+}\cdot
g)(x,t) \, \overline{(\1_{A_+}\cdot
g)(x,t)} \, dx dt
\notag
\\
&=\iint \bigl(B\circ L_\la\circ B^*\bigr)(\1_{A_+}\cdot
g)(x,t) \, \overline{(\1_{A_+}\cdot
g)(x,t)} \, dx dt \notag
\\
&\qquad+\iint \bigl(B\circ G_\la\circ B^*\bigr)(\1_{A_+}\cdot
g)(x,t) \, \overline{(\1_{A_+}\cdot
g)(x,t)} \, dx dt \notag
\\
&=I + II, \notag
\end{align}
where $L_\la$ is the integral operator with kernel equaling that of $S_\la S^*_\la$ if $|t-s|\le 1$ and 
$0$ otherwise, i.e,
\begin{multline}\label{ok}
L_\la(x,t;y,s)=
\\
\begin{cases}
\bigl(S_\la S_\la^*\bigr)(x,t;y,s)=
\eta(t/T)\eta(s/T) \bigl(\, 
\beta^2(P/\la)e^{-i(t-s)\la^{-1}\Delta_g}\bigr)(x,y), \, \,
\text{if } \, \, |t-s|\le 1,
\\
0 \, \, \, \text{otherwise}.
\end{cases}
\end{multline}
Since $p\ge 2$, it is straightforward to see that \eqref{22.3''}
yields
\begin{equation}\label{22.88}
\|L_\la \|_{L^{p'}_tL^{q'}_x \to L^p_tL^q_x}
=O(\la^{\frac 2p}).
\end{equation}

If we use this, along with H\"older's inequality and
\eqref{22.10}, we obtain for the term $I$ in \eqref{22.29}
\begin{align}\label{k301}
|I|&\le \|BL_\la B^*(\1_{A_+}\cdot g)\|_{ L^p_tL^q_x}
\cdot \| \1_{A_+}\cdot g \|_{L^{p'}_tL^{q'}_x}
\\
&\lesssim \|L_\la B^*(\1_{A_+}\cdot g)\|_{L^p_tL^q_x} 
\cdot
\|\1_{A_+}\cdot g \|_{L^{p'}_tL^{q'}_x}
\notag
\\
&\lesssim \la^{\frac2{p}} \|B^*(\1_{A_+}\cdot g) \|_{L^{p'}_tL^{q'}_x}
\cdot
\|\1_{A_+}\cdot g \|_{L^{p'}_tL^{q'}_x}
\notag
\\
&\lesssim \la^{\frac2{p}}\|g\|^2_{L^{p'}_tL^{q'}_x(A_+)}
=\la^{\frac2{p}}. \notag
\end{align}

To estimate  the other term in \eqref{22.29}, $II$, we need the 
 following kernel bound which is Proposition 4.1 in \cite{blair2023strichartz},
\begin{equation}\label{k}
|(S_\la S_\la^*)(x,t;y,s)|\le C
\la^{\frac{n-1}2} \, |t-s|^{-\frac{n-1}2}
\exp(C_M|t-s|), \quad \text{if } \, \, |t-s|\le 2T.
\end{equation}
The proof of \eqref{k} follows from arguments in B\'erard~\cite{Berard}, and also  in \cite{BSTop}, \cite{SoggeZelditchL4} and other related works, which use the Hadamard parametrix and the 
Cartan-Hadamard theorem to lift the 
calculations that will be needed up to the
universal cover $({\mathbb R}^{n-1},\tilde g)$
of $(M^{n-1},g)$.

If we choose $c_0$ small enough
so that if $C_M$ is the constant
in \eqref{k}
$$\exp(2C_MT)\le \la^{\e_1}, \quad \text{if  }
\,
T=c_0\log\la \, \, \text{and } \, \, \la \gg 1.
$$
Then, since $\eta(t)=0$ for $|t|\ge1$, it follows
from \eqref{ok} and \eqref{k} that
$$\|G_\la \|_{L^1(M^{n-1}\times \R) 
\to L^\infty (M^{n-1}\times \R)}\le C
\la^{\frac{n-1}2+{\e_1}}.
$$
As a result, since, by \eqref{22.10}, the dyadic operators $B$ are bounded
on $L^1$ and $L^\infty$, we can repeat the arguments
to estimate $I$ and use
H\"older's inequality to see that
$$|II|\le C\la^{\frac{n-1}2}\la^{\e_1}
\| \1_{A_+}\cdot g\|^2_1
\le C \la^{\frac{n-1}2}\la^{\e_1}
\|g\|^2_{L^{p'}_tL^{q'}_x} \cdot
\|\1_{A_+}\|_{L^{p}_tL^{q}_x}^2
=C \la^{\frac{n-1}2}\la^{\e_1}
\|\1_{A_+}\|_{L^{p}_tL^{q}_x}^2.
$$
If we recall the definition of $A_+$ in
\eqref{22.24}, we can estimate the last
factor:
$$\|\1_{A_+}\|_{L^{p}_tL^{q}_x}^2
\le \bigl(\la^{\frac{n-1}4+\e_1}\bigr)^{-2}
\|\tilde S_\la f\|^2_{{L^{p}_tL^{q}_x}(A_+)}.$$
Therefore, 
$$|II|\lesssim \la^{-\e_1}
\|\tilde S_\la f\|^2_{{L^{p}_tL^{q}_x}(A_+)}
\le \bigl(\tfrac12 \|\tilde S_\la f\|_{{L^{p}_tL^{q}_x}(A_+)}
\bigr)^2,
$$
assuming, as we may, that $\la$ is large enough.

If we combine this bound with the earlier one,
\eqref{k301} for $I$, we conclude that
\eqref{high'} is valid, which completes the 
proof of \eqref{22.3'} on the set $A_+$. \qed

\bigskip

\noindent{\bf 2.3.   Estimates for relatively small values: Proof of  \eqref{22.3'} on the set $A_-$.}

We now turn to the proving the ${L^{p}_tL^{q}_x}(A_-)$ estimates  in \eqref{22.3'}.
To do this we
need to borrow the bilinear estimates from \cite{blair2023strichartz}, which replies on the results from bilinear harmonic analysis in \cite{LeeBilinear} and
\cite{TaoVargasVega}.

We need to utilize a microlocal decomposition as in \cite{blair2023strichartz}.
Recall that the symbol $B(x,\xi)$ of $B$ in \eqref{22.9} is supported in a small
conic neighborhood of some $(x_0,\xi_0)\in S^*M^{n-1}$.  We may assume that its symbol has
small enough support so that we may work in a coordinate chart $\Omega$ and that
$x_0=0$, $\xi_0=(0,\dots,0,1)$ and $g_{jk}(0)=\delta^j_k$ in the local coordinates.
So, we shall assume that $B(x,\xi)=0$ when $x$ is outside a small relatively compact neighborhood
of the origin or $\xi$ is outside of a small conic neighborhood of $(0,\dots,0,1)$.  

Next, let us define the microlocal cutoffs that we shall use.   We fix a function
$a\in C^\infty_0({\mathbb R}^{2(n-2)})$ supported in $\{z: \, |z_j|\le 1, \, \, 1\le j\le 2(n-2)\}$
 which satisfies
\begin{equation}\label{m1}
\sum_{j\in {\mathbb Z}^{2(n-2)}}a(z-j)\equiv 1
.
\end{equation}
We shall use this function to build our microlocal cutoffs.
By the above, we shall focus on defining them 
 for $(y,\eta)\in S^*\Omega$ with    $y$ near the origin
 and  $\eta$ in a small conic neighborhood of $(0,\dots,0,1)$. 
We shall let
$$\Pi=\{y: \, y_{n-1}=0\}$$
be the points in $\Omega$ whose last coordinate vanishes.   Let $y'=(y_1,\dots, y_{n-2})$ and
$\eta'=(\eta_1,\dots,\eta_{n-2})$ denote the first $n-2$ coordinates of $y$ and $\eta$, respectively.
 For $y\in \Pi$ near $0$ and $\eta$ near $(0,\dots,0,1)$ we can
just use the functions $a(\theta^{-1}(y',\eta')-j)$, $j\in {\mathbb Z}^{2(n-2)}$ to obtain cutoffs of scale $\theta$. 

We can then extend the definition to a neighborhood of $(0,(0,\dots,0,1))$ by setting for $(x,\xi)\in S^*\Omega$ in this neighborhood
\begin{equation}\label{m2}
a^\theta_j(x,\xi)=a(\theta^{-1}(y',\eta')-j) \quad
\text{if } \, \, \chi_s(x,\xi)=(y',0,\eta',\eta_{n-1}) \, \, \, \text{with } \, \, \, s=d_g(x,\Pi).
\end{equation}
Here $\chi_s$ denotes geodesic flow in $S^*\Omega$.  Thus, $a^\theta_j(x,\xi)$ is constant on all geodesics
$(x(s),\xi(s))\in S^*\Omega$ with $x(0)\in \Pi$ near $0$ and $\xi(0)$ near $(0,\dots,0,1)$.   As a result,
\begin{equation}\label{m3}
a^\theta_j(\chi_s(x,\xi))=a^\theta_j(x,\xi)
\end{equation}
for $s$ near $0$ and $(x,\xi)\in S^*\Omega$ near $(0,(0,\dots,0,1))$.

We then extend the definition of the cutoffs to a conic neighborhood of $(0,(0,\dots,0,1))$  in $T^*\Omega \, \backslash \, 0$ by setting
\begin{equation}\label{m4}
a^\theta_j(x,\xi)=a^\theta_j(x,\xi/p(x,\xi)).
\end{equation}

Notice that if $(y'_j,\eta'_j)=\theta j$ and $\gamma_j$ is the geodesic in $S^*\Omega$ passing through $(y'_j,0,\eta_j)\in S^*\Omega$
with $\eta_j\in S^*_{(y'_j,0)}\Omega$ having $\eta'_j$ as its first $(n-2)$ coordinates then
\begin{equation}\label{m5}
a^\theta_j(x,\xi)=0 \quad \text{if } \, \, \,
\text{dist }\bigl((x,\xi), \gamma_j\bigr)\ge C_0\theta,
\end{equation}
for some fixed constant $C_0>0$.  Also,  $a^\theta_j$ satisfies the estimates
\begin{equation}\label{m6}
\bigl|\partial_x^\sigma \partial_\xi^\gamma a^\theta_j(x,\xi)\bigr| \lesssim \theta^{-|\sigma|-|\gamma|}, \, \, \,
(x,\xi)\in S^*\Omega
\end{equation}
related to this support property.

The $a^\theta_j$ provide ``directional'' microlocalization.  We also need a ``height'' localization since the characteristics of
the symbols of our scaled Schr\"odinger operators lie on paraboloids.  The variable coefficient operators that we shall use of course are 
adapted to our operators and are analogs of ones that are used in the study of Fourier restriction problems involving paraboloids.

To construct these, choose $b\in C^\infty_0(\R)$ supported in $|s|\le1$  satisfying $\sum_{-\infty}^\infty b(s-\ell)\equiv 1$.
 We then simply define the ``height operator'' as follows
\begin{equation}\label{m7}
A_\ell^\theta(P)=
b(\theta^{-1}\la^{-1}(P -\la\kappa^\theta_\ell)) \, 
\Upsilon(P/\la),
\quad \kappa^\theta_\ell = 1+\theta\ell, \quad
|\ell|\lesssim \theta^{-1},
\end{equation}
where if $\tilde \beta$ is as in \eqref{22.7}
\begin{equation}\label{ups}
\Upsilon\in C^\infty_0((1/10,10)) \, \, \text{satisfies } \, \, \,
\Upsilon(r)=1 \, \, \, \text{in a neighborhood of} \, \, \text{supp } \, \tilde \beta.
\end{equation}
Thus, these operators microlocalize $P$ to 
intervals of size $\approx \theta\la$ about ``heights''
$\la\kappa^\theta_\ell\approx \la$.  As we shall see below,
different ``heights'' will give rise to different
``Schr\"odinger tubes" about which the kernels of 
our microlocalization of the $\tilde \sigma_\la$
operators are highly concentrated.
Also, standard arguments as in \cite{SFIO2} show that if $A^\theta_\ell(x,y)$ is the kernel of this operator then
\begin{equation}\label{m8}
A^\theta_\ell(x,y)=O(\la^{-N}) \, \forall \, N, \quad
\text{if } \, d_g(x,y)\ge C_0\theta,
\end{equation}
for a fixed constant if $\theta\in [\la^{-\delta_0},1]$ with, as we are assuming $\delta_0<1/2$.

If $\psi(x)\in C^\infty_0(\Omega)$ equals $1$ in a neighborhood of the $x$-support of the $B(x,\xi)$ and
$A^\theta_j(x,D_x)$ is the operator with symbol
\begin{equation}\label{m9}
A^\theta_j(x,\xi)=\psi(x) a^\theta_j(x,\xi),
\end{equation}
then for $\nu=(\theta j, \theta \ell)\in \theta{\mathbb Z}^{2(n-2)+1}$ we can finally define the cutoffs that we shall use:
\begin{equation}\label{m10}
A^\theta_\nu=A^\theta_j(x,D_x)\circ A^\theta_\ell(P).
\end{equation}


Let us collect several basic facts about  the $A^\theta_\nu$ operators for later use. First,  if 
$A^\theta_\nu(x,\xi)$ and $A^\theta_{\tilde \nu}(x,\xi)$ are 
the symbols of $A^\theta_\nu$ and $A^\theta_{\tilde \nu}$, respectively, then
\begin{equation}\label{sep0}
A^\theta_\nu(x,\xi)A^\theta_{\tilde \nu}(x,\xi)\equiv 0, \quad \text{if } \, \, \, |\nu-\tilde \nu|\ge C_0 \theta,
\end{equation}
for some uniform constant $C_0$.  Also,  the principal symbol $a_\nu^\theta(x,\xi)$ of $A^\theta_\nu$ are all supported in a neighborhood of the support of $B(x,\xi)$, and satisfies
\begin{equation}\label{invariant}
a^\theta_\nu(\chi_r(x,\xi))=a^\theta_\nu(x,\xi), \, \, \text{on supp } B(x,\xi) \, \,
\text{if } \, \, |r|\le 2\delta,
\end{equation}
assuming that $\delta>0$ is small.

Besides, as operators between any  $L^p\to L^q$, $1\le p,q\le \infty$, spaces we have
\begin{equation}\label{m11}
\tilde \sigma_\la =
\sum_\nu \tilde \sigma_\la A^\theta_\nu  +O(\la^{-N}) \, \, \, \forall N,
\end{equation}
and the $A^\theta_\nu$ are almost orthogonal in the sense that we have 
\begin{equation}\label{m12}
\sum_\nu \|A^\theta_\nu h\|_{L^2_{x}}^2\lesssim \|h\|_{L^2_{x}}^2,
\end{equation}
with constants independent of $\theta\in [\la^{-\e_0},1]$.

Also, since for each $x$ the symbols vanish outside of cubes of sidelength $\theta \la$ and 
$| \partial^\gamma_\xi A^\theta_\nu(x,\xi)|=O((\la\theta)^{-|\gamma|})$, we also have that their kernels
are $O((\theta\la)^{n-1}(1+\theta\la d_g(x,y))^{-N})$ for all $N$ and so
\begin{equation}\label{mp}
\|A_\nu^\theta\|_{L^p(M)\to L^p(M)}=O(1) \quad \forall \, 1\le p\le \infty.
\end{equation}
By interpolation, \eqref{m12} and \eqref{mp} imply
\begin{equation}\label{mp1}
\|A_\nu^\theta h\|_{ \ell_\nu^pL^p(M)}\lesssim \|h\|_{L^p(M)} \quad \forall \, 2\le p\le \infty.
\end{equation}

In view of \eqref{m11} we have for $\theta_0=\la^{-\e_0}$
\begin{equation}\label{m13}
\bigl( \tilde\sigma_\la H\bigr)^2
=\sum_{\nu,\tilde \nu}
\bigl( \tilde \sigma_\la A^{\theta_0}_\nu H
\bigr)\cdot \bigl( \tilde \sigma_\la A^{\theta_0}_{\tilde\nu} H\bigr)
+O(\la^{-N}\|H\|_2^2).
\end{equation}
Recall that in $A_\nu^{\theta_0}$, $\nu \in \theta_0{\mathbb Z}^{2(n-2)+1}$ indexes a $\la^{-\e_0}$-separated set in
${\mathbb R}^{2n-3}$. Here $\e_0<\frac12$ is small constant that we shall specify later, the choice of $\e_0$ depends on the dimension $d=n-1$.

We need to organize the pairs of indices $\nu,\tilde \nu$ in \eqref{m13} as in many earlier works (see \cite{LeeBilinear} and 
\cite{TaoVargasVega}).
To this end, consider dyadic cubes,
$\tau^\theta_\mu$ in ${\mathbb R}^{2n-3}$ of sidelength $\theta=2^k\theta_0$, with $\tau^\theta_\mu$ denoting translations of the cube
$[0,\theta)^{2n-3}$ by $\mu\in \theta{\mathbb Z}^{2n-3}$.
Two such dyadic cubes of sidelength $\theta$ are said
to be {\em close}  if they are not adjacent but have
adjacent parents of length $2\theta$, and, in this case,
we write $\tau^\theta_\mu \sim \tau^\theta_{\tilde \mu}$.
We note that close cubes satisfy $\text{dist}(\tau^\theta_\mu, \tau^\theta_{\tilde \mu})\approx \theta$, and so each fixed cube has $O(1)$ cubes which are ``close'' to it.  Moreover, as noted in 
\cite[p. 971]{TaoVargasVega}, any distinct points $\nu,
\tilde \nu\in {\mathbb R}^{2n-3}$ must lie in
a unique pair of close cubes in this Whitney decomposition.  So, there must be a unique triple
$(\theta=\theta_02^k, \mu, \tilde \mu)$ such that
$(\nu,\tilde \nu)\in \tau^\theta_\mu\times \tau^\theta_{\tilde \mu}$ and $\tau^\theta_\mu\sim
\tau^\theta_{\tilde \mu}$.  We remark that by choosing
$B$ to have small support we need only consider 
$\theta=2^k\theta_0\ll 1$.

Taking these observations into account,
as in earlier works, we conclude 
that the bilinear sum \eqref{m13} can be 
organized as follows:
\begin{multline}\label{m14}
\sum_{\{k\in {\mathbb N}: \, k\ge 10 \, \, \text{and } \, 
\theta=2^k\theta_0\ll 1\}}
\sum_{\{(\mu,\tilde \mu): \, \tau^\theta_\mu
\sim \tau^\theta_{\tilde \mu}\}}
\sum_{\{(\nu,\tilde \nu)\in
\tau^\theta_\mu\times \tau^\theta_{\tilde \mu}\}}
\bigl(\tilde \sigma_\la
A^{\theta_0}_\nu H\bigr) 
\cdot \bigl(\tilde \sigma_\la
A^{\theta_0}_{\tilde \nu} H\bigr)
\\
+\sum_{(\tau,\tilde \tau)\in \Xi_{\theta_0}} 
\bigl( \tilde \sigma_\la A^{\theta_0}_\nu H\bigr) 
\cdot \bigl(\tilde \sigma_\la
A^{\theta_0}_{\tilde \nu} 
H\bigr)
,
\end{multline}
where $\Xi_{\theta_0}$ indexes the remaining pairs such
that $|\nu-\tilde \nu|\lesssim \theta_0=\la^{-\e_0}$,
including the diagonal ones where $\nu=\tilde \nu$.

The key  estimate that we require, which follows from bilinear harmonic analysis arguments, then is the following.

\begin{proposition}\label{locprop}
If $H=S_\la f$ is as in \eqref{22.2} and $(p,q)$ is as in \eqref{00.4'}, we have
\begin{equation}\label{b1}
\|\tilde \sigma_\la H\|_{L^{p}_tL^{q}_x(A_-)}
\lesssim  \bigl\|
 \tilde \sigma_\la A^{\theta_0}_\nu H
\bigr\|_{L^{p}_t\ell_\nu^{q}L^{q}_x(M^{n-1}\times [0, T])}
+\la^{\frac1{p}-}\|H\|_{L^2_{t,x}(M^{n-1}\times \R)}.
\end{equation}
\end{proposition}

The $\la^{\frac1{p}-}$  notation that we are using for the last term in \eqref{b1} denotes $\la^{\frac1{p}-\e}$ for some
unspecified $\e>0$.  Note that since $\|H\|_{L^2_{t,x}}\approx T^{1/2}$ for $H=S_\la f$ and $T\approx \log\la$ the log-loss afforded by having the last term involve this norm is more than overset by the power gain $1/p-$ of $\lambda$. \eqref{b1} is the place where we require the mixed--norm to be taken over the set $A_-$. As we shall see later in the proof, the upper bound of $\tilde \sigma_\la H$ on the set $A_-$ will allow us to fully exploit the gain from bilinear estimates.

We shall postpone the proof of Proposition~\ref{locprop} until the next section.  Let us now
see how we can use it to prove \eqref{22.3'} on the set $A_-$. Given \eqref{b1}, it suffices to show that 
 when
$M^{n-1}$ has nonpositive curvature
\begin{equation}\label{22.53}
 \bigl\|
 \tilde \sigma_\la A^{\theta_0}_\nu H
\bigr\|_{L^{p}_t\ell_\nu^{q}L^{q}_x(M^{n-1}\times [0, T])} \le C \la^{\frac1p},
\end{equation}
with $T=c_0\log\la$ for $c_0>0$ sufficiently small.

We  shall also need the following simple lemma whose proof we postpone until the 
end of this subsection.

\begin{lemma}\label{comprop}  
If $\delta>0$ in \eqref{22.6} is small enough
and
$\theta_0 =\la^{-\e_0}$ we have
for 
$B$ as in \eqref{22.9}
\begin{equation}\label{cc3}
\bigl\| \, 
B \sigma_\la A^{\theta_0}_\nu -B A^{\theta_0}_\nu
\sigma_\la \, \bigr\|_{L^2_{t,x}\to{L^{p}_tL^{q}_x}}
=O(\la^{\frac1p-\frac12+2\e_0}).
\end{equation}
\end{lemma}

By \eqref{cc3} along with the fact that $\ell^{q}\subset \ell^2$ if $q\ge 2$ , we have 
\begin{equation}\label{22.60}
    \begin{aligned}
        \bigl\| &
 \tilde \sigma_\la A^{\theta_0}_\nu   H
\bigr\|_{L^{p}_t\ell_\nu^{q}L^{q}_x}
\\
&\lesssim  \bigl\| 
B A^{\theta_0}_\nu
\sigma_\la   H
\bigr\|_{L^{p}_t\ell_\nu^{q}L^{q}_x} +\| \, 
(B \sigma_\la A^{\theta_0}_\nu -B A^{\theta_0}_\nu
\sigma_\la )H\,\|_{\ell_\nu^{2}L^{p}_tL^{q}_x} 
\\
&\lesssim 
\bigl\| 
B A^{\theta_0}_\nu
\sigma_\la   H
\bigr\|_{L^{p}_t\ell_\nu^{q}L^{q}_x}+\la^{\frac1p-\frac12+2\e_0} \| \, 
H\,\|_{\ell_\nu^{2}L^2_{t,x}} .
    \end{aligned}
\end{equation}
Since the number of choices of $\nu$ is $O(\la^{(2n-3)\e_0})$ and $H$ is independent of $\nu$, we have $\| \, 
H\,\|_{\ell_\nu^{2}L^2_{t,x}}\lesssim\la^{(n-\frac32)\e_0}\| \, 
H\,\|_{L^2_{t,x}}$. Thus if we choose $\e_0<\tfrac{1}{2n+1}$,  the second term on the right side of \eqref{22.60} is bounded by 
$\la^{\frac1{p}-}\|H\|_{L^2_{t,x}}$.


On the other hand, since we are assuming $f$ is $L^2$ normalized in \eqref{normalize},  by \eqref{22.10}, \eqref{mp1}, \eqref{22.12} and the fact that $H=S_\la f$, we have
\begin{equation}\label{22.61}
\bigl\| 
B A^{\theta_0}_\nu
\big(I-\sigma_\la\big)  H
\bigr\|_{L^{p}_t\ell_\nu^{q}L^{q}_x(M^{n-1}\times [0, T])} 
\le \bigl\| 
\big(I-\sigma_\la\big)  H
\bigr\|_{L^{p}_tL^{q}_x(M^{n-1}\times [0, T])}\le \la^{\frac1{p}}.
\end{equation}
Thus,  by \eqref{22.60} we would have \eqref{22.53} if
\begin{equation}\label{22.62}
\bigl\| 
B A^{\theta_0}_\nu
   H
\bigr\|_{L^{p}_t\ell_\nu^{q}L^{q}_x(M^{n-1}\times [0, T])}\lesssim \la^{\frac1{p}},
\end{equation}
which, by \eqref{22.10} and the fact that $H=S_\la f$, is a consequence of 
\begin{equation}\label{22.622}
\bigl\| 
 A^{\theta_0}_\nu
 S_\la f
\bigr\|_{L^{p}_t\ell_\nu^{q}L^{q}_x(M^{n-1}\times [0, T])}\lesssim \la^{\frac1{p}}.
\end{equation}

To prove \eqref{22.622}, let us define 
$$f\to \bigl(Wf \bigr)(x,t,\nu)=\eta(t/T)  \bigl(A_{\nu}^{\theta_0}\circ e^{-it\la^{-1}\Delta_g}f\bigr)(x).
$$
By applying the abstract theorem of Keel-Tao \cite{KT} and a simple rescaling argument, 
 we 
would have \eqref{22.622} if
\begin{equation}\label{22.62'}
\| Wf(t,\cdot)\|_{\ell_\nu^{2}L^{2}_{x}}
\le C \|f\|_{L^2_x},
\end{equation}
and 
\begin{equation}\label{22.63'}
\| W(t)W^*(s)G\|_{\ell_\nu^{\infty}L^{\infty}_x}
\le  C\la^{\frac{n-1}2} \, |t-s|^{-\frac{n-1}2}  \|G\|_{\ell_\nu^{1}L^{1}_x},
\end{equation}
with
\begin{align}\label{22.64}
&WW^*G(x,t,\nu)=
\\
&= \eta(t/T) \sum_{\nu'} \int_{-\infty}^\infty
 \eta(s/T) \Bigl[ \bigr( A_{\nu}^{\theta_0} e^{-i(t-s)\la^{-1}\Delta_g} (A^{\theta_0}_{\nu'})^* \bigr)G(\, \cdot \, , s,\nu')\Bigr] (x) \, ds \notag
\\
&=\sum_{\nu'} \iint K(x,t,\nu;y,s,\nu') \, G(y,s,\nu') \, dyds, \notag
\end{align}
where
\begin{equation}\label{22.65}
K(x,t,\nu; y,s,\nu') =  \eta(t/T) \bigr( A_{\nu}^{\theta_0} e^{-i(t-s)\la^{-1}\Delta_g} (A^{\theta_0}_{\nu'})^* \bigr)(x,y) \, \eta(s/T).
\end{equation}

It is not hard to check that \eqref{22.62'} follows from \eqref{mp} and the fact that $e^{-it\la^{-1}\Delta_g}$ is unitary, and \eqref{22.63'} is a consequence of the kernel estimates
\begin{equation}\label{22.65np}
|K(x,t,\nu;y,s,\nu')|\le C\la^{\frac{n-1}2} \, |t-s|^{-\frac{n-1}2},
\end{equation}
 which follows from Proposition 4.2 in \cite{blair2023strichartz}.

 Here compared with \eqref{k}, there is no $e^{CT}$ loss in the above kernel estimates due to the presence of the $A_{\nu}^{\theta_0}$ operators.
The proof of \eqref{22.65np} follows from exploiting the directional localization in the $A_{\nu}^{\theta_0}$ operators through the use of Toponogov's
triangle comparison theorem.   See section 4 in \cite{blair2023strichartz}, and also recent works \cite{BSTop} and \cite{SBLog}  for more details.
 Also
note that here we take $\theta_0=\la^{-\e_0}$, which may be much larger than $\theta_0=\la^{-\frac18}$ as  in \cite{blair2023strichartz}; however, the proof of the microlocalized kernel estimates extends to the larger $\theta_0$ similarly as long as we fix $T=c_0\log\la$ for some $c_0\ll \e_0$.

This completes the proof of  \eqref{22.3'} on the set $A_-$  up to proving
the crucial local estimates in Proposition~\ref{locprop}, which we shall postpone to the  next section.

The other task remaining to complete the proofs Theorems \ref{nonposthm}   is to 
prove the  commutator estimate that we employed:

\begin{proof}[Proof of Lemma~\ref{comprop}] The proof follows from similar arguments as in Lemma 2.7 in \cite{blair2023strichartz}, we include the details here for completeness.
Recall that 
by \eqref{22.9} the symbol $B(x,\xi)=B_\la(x,\xi)
\in S^{0}_{1,0}$ vanishes when $|\xi|$ is not
comparable to $\la$.  In particular, it vanishes if $|\xi|$ is larger than a fixed
multiple of $\la$, and it belongs to a bounded subset of $S^0_{1,0}$.
Furthermore, if $a^{\theta_0}_\nu(x,\xi)$ is the principal
symbol of our zero-order dyadic microlocal operators, we recall  that by \eqref{invariant}
we have that for $\delta>0$ small enough
\begin{equation}\label{cc2}
a^{\theta_0}_\nu(x,\xi)=a^{\theta_0}_\nu(\chi_r(x,\xi))
\quad \text{on supp } \, B_\la \, \, \, 
\text{if } \, \, |r|\le 2\delta,
\end{equation}
where $\chi_r: \, T^*M^{n-1}\, \backslash 0 \to T^*M^{n-1} \, \backslash 0$ denotes
geodesic flow in the cotangent bundle.

By Sobolev estimates for $M^{n-1}\times \R$, in order
to prove \eqref{cc3}, it suffices to show that
\begin{multline}\label{cc4'}
\Bigl\|\,
\Bigl(\sqrt{I+P^2+D_t^2} \, \, \Bigr)^{(n-1)(\frac12-\frac1{q})+\frac12-\frac1p}
\,
\bigl[ B_\la \sigma_\la A^{\theta_0}_\nu
-B_\la A^{\theta_0}_\nu \sigma_\la \bigr]
\, 
\Bigr\|_{L^2_{t,x}\to L^2_{t,x}}
\\ =
O(\la^{\frac1{p}-\frac12+2\e_0}).
\end{multline}
Note that for $(p,q)$ as in \eqref{00.4'}, $(n-1)(\frac12-\frac1{q})+\frac12-\frac1p=\frac12+\frac1p$, and 
since the symbol of $B_\la$ and $\sigma_\la$ are supported in $|\xi|, |\tau|\approx \la$, in order
to prove \eqref{cc4'}, it suffices to show that
\begin{equation}\label{cc4}
\Bigl\|\,
 B_\la \sigma_\la A^{\theta_0}_\nu
-B_\la A^{\theta_0}_\nu \sigma_\la 
\, 
\Bigr\|_{L^2_{t,x}\to L^2_{t,x}}
=
O(\la^{-1+2\e_0}).
\end{equation}

To prove this we recall that
$$\sigma_\la=(2\pi)^{-1}\tilde \beta(D_t/\la)
\int \Hat \sigma(r)
e^{ir\la^{1/2}|D_t|^{1/2}} \, e^{-irP} \, dr,
$$
and, therefore, since $e^{ir\la^{1/2}|D_t|^{1/2}}$
has $L^2\to L^2$ norm one and commutes with $B_\la$ and
$A^{\theta_0}_\nu$, and since
$\Hat \sigma(r)=0$, $|r|\ge 2\delta$, by
Minkowski's integral inequality,
we would have
\eqref{cc4} if 
\begin{equation}\label{cc5}
\sup_{|r|\le 2\delta}\, 
\Bigl\|\,
\, \tilde \beta(D_t/\la) \, 
\bigl[ B_\la e^{-irP} A^{\theta_0}_\nu
-B_\la A^{\theta_0}_\nu e^{-irP} \bigr]
\, 
\Bigr\|_{L^2_{t,x}\to L^2_{t,x}}
=
O(\la^{-1+2\e_0}).
\end{equation}

Next, to be able to use Egorov's theorem, we write
$$\bigl[ B_\la e^{-irP} A^{\theta_0}_\nu
-B_\la A^{\theta_0}_\nu e^{-irP} \bigr]
=B_\la
\, \bigl[
(e^{-irP}A^{\theta_0}_\nu e^{irP}) - B_\la A^{\theta_0}_\nu]
\circ e^{-irP}.
$$
Since $e^{-irP}$ also has $L^2$-operator norm one, we
would obtain \eqref{cc5} from
\begin{equation}\label{cc6}
\Bigl\|\,
\, \tilde \beta(D_t/\la) \, 
B_\la
\, \bigl[
(e^{-irP}A^{\theta_0}_\nu e^{irP}) - A^{\theta_0}_\nu\bigr]
\, 
\Bigr\|_{L^2_{t,x}\to L^2_{t,x}} 
=
O(\la^{-1+2\e_0}),
\end{equation}
which is a simple consequence of the   Egorov's theorem, see e.g., Taylor~\cite[\S VIII.1]{TaylorPDO} and Lemma 2.7 in \cite{blair2023strichartz} for more details.
\end{proof}

\newsection{Proof of Proposition~\ref{locprop}}

 Let us fix $p=2, q=q_e=\frac{2(n-1)}{n-3}$,
we shall first give the arguments for $d=n-1\ge 4$, and later modify it for $n-1=3$. The case $d=n-1=2$ follows from similar arguments, which we shall briefly discuss at the end of this section.

To prove \eqref{b1},  the strategy is similar to \cite{blair2023strichartz}, which is related to the ideas  in Blair
and Sogge \cite{SBLog} and earlier works, especially
Tao, Vargas and Vega~\cite{TaoVargasVega} and Lee~\cite{LeeBilinear}.

We first note that if $\delta$ as in \eqref{22.6} is small enough we have
\begin{equation}\label{b2}
\tilde \sigma_\la -
\sum_\nu  \tilde \sigma_\la \Atn=R_\la,
\, \,  \text{where } \, \, 
\|R_\la H\|_{L^\infty_{t,x}}\lesssim \la^{-N}
\|H\|_{L^2_{t,x}} \, \, \forall N.
\end{equation}
Thus, we have
\begin{equation}\label{b3}
\bigl(\, \tilde \sigma_\la H\, \bigr)^2
=\sum_{\nu,\tilde \nu}
\bigl(\tilde \sigma_\la \Atn H\bigr)\cdot
\bigl(\tilde \sigma_\la \Atnt H\bigr)
+O(\la^{-N} \|H\|_{L^2_{t,x}}^2) \, \, \forall \, N.
\end{equation}

As in earlier works, let
\begin{equation}\label{b4}
\diag(H) = \sum_{(\nu,\tilde \nu)\in \xid}
\bigl(\tilde \sigma_\la \Atn H\bigr)\cdot
\bigl(\tilde \sigma_\la \Atnt H\bigr),
\end{equation}
and
\begin{equation}\label{b5}
\far(H)=\sum_{(\nu,\tilde \nu)\notin \xid}
\bigl(\tilde \sigma_\la \Atn H\bigr)\cdot
\bigl(\tilde \sigma_\la \Atnt H\bigr)
+O(\la^{-N} \|H\|_{L^2_{t,x}}^2),\end{equation}
with the last term denoting the error term in \eqref{b3}.
Thus,
\begin{equation}\label{b6}
\bigl(\tilde \sigma_\la H\bigr)^2
=\diag(H)+\far(H).
\end{equation}
Here, the summation in $\diag(H)$ is over near diagonal pairs $(\nu,\tilde \nu)\in \xid$ as in \eqref{m14}.  In particular we
have $|\nu-\tilde \nu|\le C\theta_0$ for some uniform constant
as $\nu,\tilde \nu$ range over $\theta_0{\mathbb Z}^{(2n-3)}$.
The other term $\far(H)$ is the remaining pairs, which include many which are far from the diagonal.  This sum will provide the contribution to the last term in
\eqref{b1}. 

The two types of terms here are treated differently,
as in analyzing parabolic restriction problems or 
spectral projection estimates. 

We shall treat the first term in the right of \eqref{b6} as in \cite{BHSsp} and \cite{SBLog} by using a variable coefficient variant of
Lemma 6.1 in 
\cite{TaoVargasVega} (see also Lemma 4.2 in \cite{SBLog}):

\begin{lemma}\label{blemma}  If $\diag(H)$ is as in
\eqref{b6} and $d=n-1\ge4$, then we have the uniform bounds
\begin{equation}\label{b7}
\|\diag(H)\|_{L^{1}_tL^{q_e/2}_x}
\lesssim
\| \tilde \sigma_\la
\Atn H\|^2_{L^{2}_t\ell_\nu^{q_e}L^{q_e}_x}\,
+O(\la^{1-}\|H\|_{L^2_{t,x}}^2).
\end{equation}
\end{lemma}

We also need the following estimate for
$\far(H)$ which is a consequence of the  bilinear
oscillatory integral estimates of Lee~\cite{LeeBilinear}
and arguments of Blair and Sogge in \cite{BlairSoggeRefined}, \cite{blair2015refined} and \cite{SBLog}.

\begin{lemma}\label{leelemma}
If $\far(H)$ is as in \eqref{b5}, and, as above
$\theta_0=\la^{-\e_0}$, then for all $\e>0$ we have
for $H=S_\la f$
\begin{equation}\label{b8}
\iint |\far(H)|^{q/2}\, dxdt \lesssim_{\e}
\la^{1+\e} \, \bigl(\la^{1-\e_0}\bigr)^{\frac{n-1}2
(q-\frac{2(n+1)}{n-1}))} \, \|H\|_{L^2_{t,x}}^{q}, \quad
\text{if } \, q=\tfrac{2(n+2)}n.
\end{equation}
\end{lemma}

Let us postpone the proofs of these two lemmas for a
bit and show how they can be used to obtain
Proposition~\ref{locprop} if $d=n-1\ge4$.

If we let $q=\tfrac{2(n+2)}n$ as in Lemma~\ref{leelemma}, we note that
$q<q_e$ and also
\begin{equation}| \tilde \sigma_\la H \cdot
 \tilde \sigma_\la H|^{q_e}
\le 2^{q/2}
\, |\tilde \sigma_\la H \cdot
 \tilde \sigma_\la H|^{\frac{q_e-q}2}
\cdot \bigl( \,
|\diag(H)|^{q/2}+|\far(H)|^{q/2}\, \bigr).
\end{equation}
Thus,
\begin{align}\label{b9}
\| &\tilde \sigma_\la H
\|_{L^{2}_tL^{q_e}_x(A_-)}^2
=\int \Big(\int\bigl|  \tilde \sigma_\la H \cdot
 \tilde \sigma_\la H\bigr|^{q_e/2} \, dx\Big)^{\frac{2}{q_e}}dt
\\
&\lesssim \notag
\int\Big(\int | \tilde \sigma_\la H \cdot
 \tilde \sigma_\la H|^{\frac{q_e-q}2} \, 
|\diag(H)|^{q/2} \, dx\Big)^{\frac{2}{q_e}}dt
\\
&\quad +\int\Big(\int | \tilde \sigma_\la H \cdot
\tilde \sigma_\la H|^{\frac{q_e-q}2} \, 
|\far(H)|^{q/2} \, dx\Big)^{\frac{2}{q_e}}dt  =I+II.
\notag
\end{align}

To estimate $II$, first note that $\tilde \sigma_\la
H=\tilde S_\la f$ if $H=S_\la f$, we have 
\begin{equation}\label{ii1}
\begin{aligned}
     II &\lesssim \|\tilde S_\la f\|_{L^\infty(A_-)}^{\frac{2(q_e-q)}{q_e}} \cdot 
\|\far(H)\|_{L^{\frac{q}{q_e}}_{t}L^{\frac q2}_x}^{\frac{q}{q_e}} \\
&\lesssim T^{(\frac{q_e}{q}-\frac{2}{q})\cdot\frac{q}{q_e}} \|\tilde S_\la f\|_{L^\infty(A_-)}^{\frac{2(q_e-q)}{q_e}} \cdot 
\|\far(H)\|_{L^{\frac{q}{2}}_{t}L^{\frac q2}_x}^{\frac{q}{q_e}},
\end{aligned}
\end{equation}
where in the second line we used H\"older's inequality on the time interval $[0,T]$. By \eqref{b8} and  the ceiling
for $A_-$, we have 
\begin{equation}\label{far}
\begin{aligned}
    II \le  T^{(\frac{q_e}{q}-\frac{2}{q})\cdot\frac{q}{q_e}} \la ^{(\frac{n-1}4+\e_1)(\frac{2(q_e-q)}{q_e})}
\cdot \Big(\la^{1+\e} \, \bigl(\la^{1-\e_0}\bigr)^{\frac{n-1}2
(q-\frac{2(n+1)}{n-1})} \Big)^{\frac{2}{q_e}}\|H\|^{\frac{2q}{q_e}}_{L^2_{x}}.
\end{aligned}
\end{equation}
If we take $\e_0, \e_1$ and $\e$ to be small enough, e.g., $\e=\e_1=\frac1{100}$ and $\e_0=\frac{1}{2n+2}$, it is straightforward to check that 
\begin{equation}\label{far1}
\begin{aligned}
    II 
=O(\la^{1-}\|H\|_{L^2_{t,x}}^{\frac{2q}{q_e}})=O(\la^{1-}\|H\|_{L^2_{t,x}}^{2}).
\end{aligned}
\end{equation}
Here we also  used the fact that 
$\|H\|_{L^2_{t,x}}^{2}$ dominates 
$\|H\|_{L^2_{t,x}}^{\frac{2q}{q_e}}$ since $q_e>q$ and
$\|H\|_{L^2_{t,x}}\approx T$ since $H=S_\la f$, $\|f\|_2=1$
and $e^{-it\la^{-1}\Delta_g}$ is a unitary operator on $L^2_x$.  

Consequently, we just need to see that $I^{1/2}$ is
dominated by the other term in the right side of this inequality
To estimate this term we
use H\"older's inequality followed by Young's 
inequality and Lemma~\ref{blemma} to see that
\begin{align*}
I=&
\int\Big(\int | \tilde \sigma_\la H \cdot
 \tilde \sigma_\la H|^{\frac{q_e-q}2} \, 
|\diag(H)|^{\frac q2} \, dx\Big)^{\frac{2}{q_e}}dt \\
\le & \int\Big( \| \tilde \sigma_\la H \cdot
 \tilde \sigma_\la H\|_{L_x^{\frac{q_e}{2}}}^{\frac{q_e-q}2} \, 
\|\diag(H)\|_{L_x^{\frac{q_e}{2}}}^{\frac q2} \,\Big)^{\frac{2}{q_e}}dt \\
\le &\| \tilde \sigma_\la H \cdot
 \tilde \sigma_\la H\|_{L^1_tL_x^{\frac{q_e}{2}}}^{\frac{q_e-q}{q_e}} \, 
\|\diag(H)\|_{L^1_tL_x^{\frac{q_e}{2}}}^{\frac{q}{q_e}} \\
\le &\tfrac{q_e-q}{q_e}\| \tilde \sigma_\la H \|^2_{L^2_tL_x^{q_e}} +\tfrac{q}{q_e}
\|\diag(H)\|_{L^1_tL_x^{\frac{q_e}{2}}}\\
\le &\tfrac{q_e-q}{q_e}\| \tilde \sigma_\la H \|^2_{L^2_tL_x^{q_e}} +C(
\| \tilde \sigma_\la
\Atn H\|^2_{L^{2}_t\ell_\nu^{q_e}L^{q_e}_x}\,
+\la^{1-}\|H\|_{L^2_{t,x}}^2).
\end{align*}
Since $\tfrac{q_e-q}{q_e}<1$, the first term
in the right can be absorbed in the left
side of \eqref{b9}, and this, along with
the estimate for $II$ above yields \eqref{b1}.

Thus, if we can prove Lemma~\ref{blemma} and
Lemma~\ref{leelemma}, the proof of Proposition~\ref{locprop}
for $d=n-1\ge4$
 will be complete.

\medskip

\begin{proof}[\textbf{Proof of Lemma~\ref{blemma}}]

To prove \eqref{b7}, let us first define slightly wider microlocal
cutoffs by setting
\begin{equation}\label{tanu}
    \tAtn =\sum_{|\mu-\nu|\le C_0 \theta_0}A^{\theta_0}_\mu.
\end{equation}
We can fix $C_0$ large enough so that
\begin{equation}\label{bb10}
\|\Atn-\Atn\tAtn\|_{L^p_x\to L^p_x}=O(\la^{-N})
\, \, \forall \, N\, \, \text{if } \,
\, 1\le p\le \infty.
\end{equation}
Also, like the original $\Atn$ operators
the $\tAtn$ operators are almost orthogonal
\begin{equation}\label{bb11}
\sum_\nu \|\tAtn h\|^2_{L^2_x}\lesssim
\|h\|_{L^2_x}^2.
\end{equation}

Since by \eqref{22.10} and \eqref{qu1}, we have
\begin{equation}\label{tsigma}
    \|\tilde \sigma_\la F\|_{L^{2}_tL^{q_e}_x(M^{n-1}\times [0, T])}\le C\la^{\frac1{2}}\|F\|_{L^2_{t,x}},
\end{equation}
we conclude that, in order to prove \eqref{b7}, we may replace $\diag(H)$ by $\tdiag(H)$ where
the latter is defined by the analog of 
\eqref{b4} with $\Atn$ and $\Atnt$ 
replaced by $\Atn \tAtn$ and $\Atnt
\tAtnt$, respectively.

So, it suffices to prove
\begin{equation}\label{b12}
    \begin{aligned}
        \bigl\|\sum_{(\nu,\tilde \nu)\in \xid}
(\tilde \sigma_\la
\Atn\tAtn H)\cdot 
(\tilde \sigma_\la
\Atnt\tAtnt &H)\bigr\|
_{L^{1}_tL^{q_e/2}_x} \\
&\lesssim
\| \tilde \sigma_\la
\Atn H\|^2_{L^{2}_t\ell_\nu^{q_e}L^{q_e}_x}\,
+O(\la^{1-}\|H\|_{L^2_{t,x}}^2).
    \end{aligned}
\end{equation}

We shall need the following variant of  \eqref{cc3},
\begin{equation}\label{b13}
\|[\, \tilde \sigma_\la
\Atn-\Atn\tilde \sigma_\la\, ]F\|_{L^{2}_tL^{q_e}_x}
\lesssim \la^{2\e_0}\|F\|_{L^2_{t,x}}.
\end{equation}
This follows from the proof of Lemma~\ref{comprop}, \eqref{cc3} and the fact that the commutator
$[B,A^{\theta_0}_\nu]$ is bounded on $L^{q_e}_x(M^{n-1})$ with norm
$O(\la^{-1+\e_0})$. 
By \eqref{bb11} and \eqref{b13} we 
would have \eqref{b12} if we could show that
\begin{equation}\label{b14}
    \begin{aligned}
        \bigl\|\sum_{(\nu,\tilde \nu)\in \xid}
(\Atn\tilde \sigma_\la
\tAtn H)\cdot 
(\Atnt\tilde \sigma_\la
\tAtnt &H)\bigr\|
_{L^{1}_tL^{q_e/2}_x} \\
&\lesssim
\| \tilde \sigma_\la
\Atn H\|^2_{L^{2}_t\ell_\nu^{q_e}L^{q_e}_x}\,
+O(\la^{1-}\|H\|_{L^2_{t,x}}^2).
    \end{aligned}
\end{equation}

Since the $\Atn$ operators are time independent, we claim that, it suffices to show that for arbitrary $h_{\nu}, h_{\tilde \nu}$, which may depend on $\nu$ and $\tilde\nu$,
\begin{equation}\label{b14'}
    \begin{aligned}
        \bigl\|\sum_{(\nu,\tilde \nu)\in \xid}
\Atn h_{ \nu}\cdot 
\Atnt h_{\tilde \nu}\bigr\|
_{L^{q_e/2}_x} \lesssim &
\Bigl(
\sum_{(\nu,\tilde \nu)\in \xid}
\| \Atn  h_{ \nu}\cdot 
\Atnt  h_{\tilde \nu}\|_{L^{q_e/2}_{x}}^{q_e/2}\,
\Bigr)^{2/q_e} \\
&\qquad
+O\big(\la^{-N}\sum_{(\nu,\tilde \nu)\in \xid}\|h_{\nu}\|_{L^1_x}\|h_{\tilde \nu}\|_{L^1_x}\big),\,\,\,\forall N.
    \end{aligned}
\end{equation}

To verify the claim, note that if we take 
$ h_{ \nu}=\tilde\sigma_\la
\tAtn H$ and $ h_{\tilde \nu}=\tilde\sigma_\la
\tAtnt H$, \eqref{b14'} implies 
\begin{equation}\label{b14''}
    \begin{aligned}
        \bigl\|\sum_{(\nu,\tilde \nu)\in \xid}&
(\Atn\tilde \sigma_\la
\tAtn H)\cdot 
(\Atnt\tilde \sigma_\la
\tAtnt H)\bigr\|
_{L^{1}_tL^{q_e/2}_x} \\
&\lesssim \int
\Bigl(
\sum_{(\nu,\tilde \nu)\in \xid}
\| (\Atn\tilde \sigma_\la
\tAtn H)\cdot 
(\Atnt\tilde \sigma_\la
\tAtnt H)\|_{L^{q_e/2}_{x}}^{q_e/2} \,
\Bigr)^{2/q_e} dt\, \\&\qquad
+\la^{-N}\int \sum_{(\nu,\tilde \nu)\in \xid}\|(\Atn\tilde \sigma_\la
\tAtn H)\|_{L^1_x}\|(\Atnt\tilde \sigma_\la
\tAtnt H)\|_{L^1_x} \,dt. \\
&\lesssim  \| \Atn\tilde\sigma_\la 
\tAtnt H\|^2_{L^{2}_t\ell_\nu^{q_e}L^{q_e}_x}\,
+O(\la^{1-}\|H\|_{L^2_{t,x}}^2).
    \end{aligned}
\end{equation}
Here  we used the fact that for fixed $\nu$, the number of choices of $\tilde\nu$ is finite, and for 
 each pair $(\nu,\tilde \nu)\in \xid$
\begin{multline}
  \int \|(\Atn\tilde \sigma_\la
\tAtn H)\|_{L^1_x}\|(\Atnt\tilde \sigma_\la
\tAtnt H)\|_{L^1_x}\, dt \\ \le \int \|(\Atn\tilde \sigma_\la
\tAtn H)\|_{L^2_x}\|(\Atnt\tilde \sigma_\la
\tAtnt H)\|_{L^2_x}\, dt=O(\|H\|_{L^2_{t,x}}^2).
\end{multline}
If we repeat earlier arguments and use
\eqref{bb10} again, we conclude that the first term in the 
right side of \eqref{b14''} is dominated by the right side of \eqref{b14}.

Thus, it remains to prove \eqref{b14'}. By duality, 
if we take
$r=(q_e/2)'$ so that $r$ is the conjugate
exponent for $q_e/2$, \eqref{b14'} is equivalent to 
\begin{equation}\label{b15}
    \begin{aligned}
       \Bigl| \sum_{(\nu,\tilde \nu)\in \xid}&
\iint \Atn h_{ \nu}\cdot 
\Atnt h_{\tilde \nu} \, \cdot G \, dx\, \Bigr|  \\ &\lesssim 
\Bigl(
\sum_{(\nu,\tilde \nu)\in \xid}
\| \Atn  h_{ \nu}\cdot 
\Atnt  h_{\tilde \nu}\|_{L^{q_e/2}_{x}}^{q_e/2}\,
\Bigr)^{2/q_e} 
\\&\qquad\qquad\qquad+O\big(\la^{-N}\sum_{(\nu,\tilde \nu)\in \xid}\|h_{\nu}\|_{L^1_x}\|h_{\tilde \nu}\|_{L^1_x}\big)\,\,\, \text{if } \, \|G\|_{L^r_{x}}=1.
    \end{aligned}
\end{equation}


To prove \eqref{b15},
note that if $x$ and $\nu$ are fixed and 
$\xi \to \Atn(x,\xi)$ does not vanish identically, then this function of $\xi$ is supported in a cube $Q^{\theta_0}_\nu(x)\subset {\mathbb R}^{n-1}_\xi$ of sidelength $\approx \la^{1-\e_0}$.
The cubes can be chosen so that, if $\eta_\nu(x)$ is its center, then $\partial^\gamma_x\eta_\nu(x)=O(\la)$ for all 
multi-indices $\gamma$.  Keeping this in mind it
is straightforward to construct for every
pair $(\nu,\tilde \nu)\in \xid$ symbols
$b_{\nu,\tilde \nu}(x,\xi)$ belonging to 
a bounded subset of $S^0_{1-\e_0,\e_0}$ satisfying
\begin{equation}\label{p1}
b_{\nu,\tilde \nu}(x,\eta)=1 \, \, \text{if } \, \,
\text{dist}\bigl(\eta, \, \text{supp}_\xi \Atn(x,\xi) \, +\, \text{supp}_\xi \Atnt(x,\xi)\bigr)
\le \la^{1-\e_0},
\end{equation}
with ``$+$'' denoting the algebraic sum.  Using this and a simple integration by parts argument shows that for every pair $(\nu,\tilde \nu)\in \xid$
\begin{equation}\label{p2}
\bigl\| (I-b_{\nu,\tilde \nu}(x,D))\bigl[
\Atn h_{ \nu} \cdot \Atnt h_{\tilde \nu}]\bigr] \bigr\|_{L^\infty_x}
\le C_N \la^{-N}\|h_{ \nu}\|_{L^1_x}\|h_{\tilde \nu}\|_{L^1_x}, \quad
\forall \, N.
\end{equation}
The symbols can also be chosen so that
$b_{\nu_1,\tilde \nu_1}(x,\xi)$ and
$b_{\nu_2,\tilde \nu_2}(x,\xi)$ have disjoint
supports if $(\nu_j,\tilde \nu_j)\in \xid$,
$j=1,2$ and $\min(|(\nu_1-\nu_2,\tilde \nu_1
-\tilde \nu_2)|, \, |(\nu_1-\tilde \nu_2,
\tilde \nu_1-\nu_2)|)\ge C_2 \theta_0$ with 
$C_2$ being a fixed constant independent of 
$\la$ since all pairs in $\xid$ are nearly
diagonal.  Due to this, the adjoints, 
$b^*_{\nu,\tilde \nu}(x,D)$ are almost orthogonal 
in the sense that we have the uniform bounds
\begin{equation}\label{p3}
\sum_{(\nu,\tilde \nu)\in \xid}
\|b^*_{\nu,\tilde \nu}(x,D)h\|^2_{L^2_x}\lesssim
\|h\|^2_{L^2_x}.
\end{equation}
Since $\text{supp}_\xi \Atn(x,\xi) \, +\, \text{supp}_\xi \Atnt(x,\xi)$ is contained
in a cube of sidelength $\approx \la^{1-\e_0}$
and can be chosen to have center $\eta_{\nu,\tilde \nu}(x)$ satisfying 
$\partial^\gamma_x \eta_{\nu,\tilde \nu}(x)
=O(\la)$, we can furthermore assume that
we have the uniform bounds
\begin{equation}\label{p4}
\sup_{(\nu,\tilde \nu)\in \xid}
\|b^*_{\nu,\tilde \nu}(x,D)h\|_{L^\infty_x}
\lesssim \|h\|_{L^\infty_x}.
\end{equation}

We have now set up our variable coefficient 
version of the simple argument in 
\cite{TaoVargasVega} that will allow us
to obtain \eqref{b15}.  First, by \eqref{p2}, it suffices to estimate
the left side of \eqref{b15} with $G$ replaced by $ b^*_{\nu,\tilde \nu}(x,D)G $. By H\"older's inequality, 
\begin{multline}\label{b16}
\Bigl| \sum_{(\nu,\tilde \nu)\in \xid}
\iint (\Atn  h_{ \nu}\cdot 
\Atnt  h_{\tilde \nu} \, \cdot \bigl( b^*_{\nu,\tilde \nu}(x,D)G \bigr)\, dx\, \Bigr|
\\
\le 
\Bigl(
\sum_{(\nu,\tilde \nu)\in \xid}
\| \Atn  h_{ \nu}\cdot 
\Atnt  h_{\tilde \nu}\|_{L^{q_e/2}_{x}}^{q_e/2}\,
\Bigr)^{2/q_e}
\cdot \Bigl(
\sum_{(\nu,\tilde \nu)\in \xid}
\| b^*_{\nu,\tilde \nu}(x,D)G\|_{L^{r}_{x}}^r
\Bigr)^{1/r}.
\end{multline}

Note that when $n-1\ge 4$, $q_e\in (2,4]$, thus $r\in [2,\infty)$. So, if we use \eqref{p3},
\eqref{p4} and an interpolation argument
we conclude that
\begin{equation}\label{almost orthgonal}
   \Bigl(
\sum_{(\nu,\tilde \nu)\in \xid}
\| b^*_{\nu,\tilde \nu}(x,D)G\|_{L^{r}_{x}}^r
\Bigr)^{1/r}=O(1),
\end{equation}
for $G$ as in \eqref{b15}. 
\end{proof}

\begin{proof}[\textbf{Proof of Lemma~\ref{leelemma}}]
To prove this let 
$\alpha\in C^\infty_0((-1,1))$ with $\alpha\equiv 1 $ in $(-\frac12, \frac12)$, then if $\alpha_m(t)=\alpha(t-m)$, up to a $\log\la$ loss, it suffices to show that for $m=1,2,\dots, T$
\begin{equation}\label{b8'}
\iint |\alpha_m(t)\far(H)|^{q/2}\, dxdt \lesssim_{\e}
\la^{1+\e} \, \bigl(\la^{1-\e_0}\bigr)^{\frac{n-1}2
(q-\frac{2(n+1)}{n-1})} \,  \|H\|_{L^2_{t,x}}^{q}.
\end{equation}

This  follow from Lemma 3.2 in \cite{blair2023strichartz}, where we replace
$\la^{7/8}$ there by $\la^{1-\e_0}$. The proof uses bilinear
oscillatory integral estimates of Lee~\cite{LeeBilinear}
and arguments of previous work of Blair and Sogge in \cite{BlairSoggeRefined}, \cite{blair2015refined} and \cite{SBLog}. And this is also where we used the condition on $\delta, \delta_0$ in \eqref{22.6}, in order to apply the bilinear oscillatory integral theorems of Lee ~\cite{LeeBilinear}, see section 3 in \cite{blair2023strichartz} for more details. 

The constants in the right side of \eqref{b8'} is better than applying  classical linear estimate.  
Actually one can also rewrite the right side as 
\begin{equation}\label{b8'''}
\la^{\frac{n-1}{2}(\frac q2-1)}\la^{\e} \, \bigl(\la^{\frac12-\e_0}\bigr)^{\frac{n-1}2
(q-\frac{2(n+1)}{n-1})} \, \|H\|_{L^2_{t,x}}^{q},
\end{equation}
thus if $\e_0\le 1/2$, we have a gain compared with the bound $\la^{\frac{n-1}{2}(\frac q2-1)}$, which follows from applying  linear estimates.
\end{proof}
\medskip

\noindent{\bf Modified arguments for $n-1=3$}.

When $n-1=3$, note that $q_e=\tfrac{2(n-1)}{n-3}=6$, thus $r=(q_e/2)'=\frac32$, which means that we do not have \eqref{almost orthgonal} and thus Lemma~\ref{blemma} in this case, so we need to modify the arguments above. If we repeat the previous arguments, it suffices to estimate the first term
\begin{equation}\label{bn2}
I=
\int\Big(\int | \tilde \sigma_\la H \cdot
 \tilde \sigma_\la H|^{\frac{q_e-q}2} \, 
|\diag(H)|^{\frac q2} \, dx\Big)^{\frac{2}{q_e}}dt,
\end{equation}
where as in \eqref{b4},
\begin{equation}\nonumber
\diag(H) = \sum_{(\nu,\tilde \nu)\in \xid}
\bigl(\tilde \sigma_\la \Atn H\bigr)\cdot
\bigl(\tilde \sigma_\la \Atnt H\bigr).
\end{equation}
For later use, note that when $n-1=3$, it is not hard to check that $ \,(\nu,\tilde \nu)\in
\Xi_{\theta_0}$ implies $|\nu-\tilde\nu|\le 2^{13}\theta_0$.

Let us define $$T_\nu H=\sum_{\tilde\nu: \,(\nu,\tilde \nu)\in
\Xi_{\theta_0}}( \tilde\sigma_\la A^{\theta_0}_\nu H )(\tilde\sigma_\la A^{\theta_0}_{\tilde\nu} H),$$
and write
\begin{equation}\label{b255}
\begin{aligned}
    ( \diag(H))^{2} &=\big(\sum_\nu T_\nu H\big)^2
    \\&= \sum_{\nu_1, \nu_2}  T_{\nu_1} HT_{\nu_2} H.
    \end{aligned}
\end{equation}

Now we can employ another Whitney-type decomposition, more explicitly,  the
sum in \eqref{b255} can be organized as
\begin{equation}\label{org}
\begin{aligned}
    &\big( \sum_{\{k\in {\mathbb N}: \, k\ge 20 \, \, \text{and } \, 
\theta=2^k\theta_0\ll 1\}} \,  \, 
\sum_{\{(\mu_1, \mu_2): \, \tau^\theta_{\mu_1}
\sim \tau^\theta_{ \mu_2}\}}
\sum_{\{(\nu_1,\nu_2)\in
\tau^\theta_{\mu_1}\times \tau^\theta_{\mu_2}\}}
+\sum_{(\nu_1,  \nu_2)\in \overline{\Xi}_{\theta_0}}
\big)T_{\nu_1} hT_{\nu_2} h \\
&={\overline\Upsilon^{\text{far}}}(H)+{\overline\Upsilon^{\text{diag}}}(H),
\end{aligned}
\end{equation}
where $\tau^\theta_{\mu_1}\sim \tau^\theta_{ \mu_2}$ means they are adjacent cubes of distance $\approx \theta$,  $\overline{\Xi}_{\theta_0}$  denotes the remaining pairs not included
in the first sum. Here the {\em{diagonal}} set $\overline{\Xi}_{\theta_0}$ is much larger than ${\Xi}_{\theta_0}$, and 
it is not hard to check that $ \,(\nu_1, \nu_2)\in
\overline{\Xi}_{\theta_0}$ implies $|\nu_1-\nu_2|\le 2^{23}\theta_0$.

By \eqref{org}, we have 
\begin{equation}\label{bn2'}
\begin{aligned}
I=\int&\Big(\int | \tilde \sigma_\la H \cdot
 \tilde \sigma_\la H|^{\frac{q_e-q}2} \, 
|\diag(H)|^{\frac{q}2} \, dx\Big)^{\frac{2}{q_e}}dt 
\\
&\lesssim 
\int\Big(\int | \tilde \sigma_\la H \cdot
 \tilde \sigma_\la H|^{\frac{q_e-q}2} \, 
|{\overline\Upsilon^{\text{diag}}}(H)|^{q/4} \, dx\Big)^{\frac{2}{q_e}}dt
\\
&\quad +\int\Big(\int | \tilde \sigma_\la H \cdot
\tilde \sigma_\la H|^{\frac{q_e-q}2} \, 
|{\overline\Upsilon^{\text{far}}}(H)|^{q/4} \, dx\Big)^{\frac{2}{q_e}}dt 
\\
&=A+B.
\end{aligned}
\end{equation}

For the diagonal term $A$, note that when $n-1=3$, $\frac{q_e}{4}\in [1,2]$, if we repeat the proof of Lemma~\ref{blemma}, it is not hard to show the following analog of \eqref{b7}
\begin{equation}\label{b77}
\|{\overline\Upsilon^{\text{diag}}}(H)\|^{\frac12}_{L^{\frac12}_tL_x^{\frac{q_e}{4}}}
\lesssim
\| \tilde \sigma_\la
\Atn H\|^2_{L^{2}_t\ell_\nu^{q_e}L^{q_e}_x}\,
+O(\la^{1-}\|H\|_{L^2_{t,x}}^2).
\end{equation}
To prove this, we just need to define the auxiliary  operator $b_{\nu_1, \nu_2}(x,D)$ in \eqref{p1} such that its frequency support is essentially the algebraic sum of the frequency support of  4 
nearby $A^{\theta_0}_{\nu}(x,D)$ operators, and the remaining arguments can be carried over in the same way.

By \eqref{b77}, 
if we use H\"older's inequality followed by Young's inequality
\begin{align*}
A=&
\int\Big(\int | \tilde \sigma_\la H \cdot
 \tilde \sigma_\la H|^{\frac{q_e-q}2} \, 
|{\overline\Upsilon^{\text{diag}}}(H)|^{\frac q4} \, dx\Big)^{\frac{2}{q_e}}dt \\
\le & \int\Big( \| \tilde \sigma_\la H \cdot
 \tilde \sigma_\la H\|_{L_x^{\frac{q_e}{2}}}^{\frac{q_e-q}2} \, 
\|{\overline\Upsilon^{\text{diag}}}(H)\|_{L_x^{\frac{q_e}{4}}}^{\frac q4} \,\Big)^{\frac{2}{q_e}}dt \\
\le &\| \tilde \sigma_\la H \cdot
 \tilde \sigma_\la H\|_{L^1_tL_x^{\frac{q_e}{2}}}^{\frac{q_e-q}{q_e}} \, 
\|{\overline\Upsilon^{\text{diag}}}(H)\|_{L^{\frac12}_tL_x^{\frac{q_e}{4}}}^{\frac{q}{2q_e}} \\
\le &\tfrac{q_e-q}{q_e}\| \tilde \sigma_\la H \|^2_{L^2_tL_x^{q_e}} +\tfrac{q}{q_e}
\|{\overline\Upsilon^{\text{diag}}}(H)\|^{\frac12}_{L^{\frac12}_tL_x^{\frac{q_e}{4}}}\\
\le &\tfrac{q_e-q}{q_e}\| \tilde \sigma_\la H \|^2_{L^2_tL_x^{q_e}} +C(
\| \tilde \sigma_\la
\Atn H\|^2_{L^{2}_t\ell_\nu^{q_e}L^{q_e}_x}\,
+\la^{1-}\|H\|_{L^2_{t,x}}^2).
\end{align*}
which is desired since the first term
in the right can be absorbed to the left
side of \eqref{b9}.

To control the off-diagonal term $B$, first we note that 
by the  Cauchy-Schwarz inequality, 
\begin{equation}\label{cauchy}
|T_{\nu_1} hT_{\nu_2} H|\le C\big(\sum_{\tilde\nu_1: \,|\tilde\nu_1-\nu_1|\le 2^{13}\theta_0}| \tilde\sigma_\la A^{\theta_0}_{\nu_1} H|^2\big)\big(\sum_{\tilde\nu_2: \,|\tilde\nu_2-\nu_2|\le 
 2^{13}\theta_0}|\tilde\sigma_\la A^{\theta_0}_{\nu_2}  H|^2\big).
\end{equation}
Since for fixed $\nu_1, \nu_2$, the number of choices of $\tilde\nu_1,\tilde\nu_2$ is finite, \eqref{cauchy} implies that 
\begin{equation}\label{b2555}
\begin{aligned}
  \sum_{\{(\mu_1, \mu_2): \, \tau^\theta_{\mu_1} 
\sim \tau^\theta_{ \mu_2}\}}&
\sum_{\{(\nu_1, \nu_2)\in
\tau^\theta_{\mu_1}\times \tau^\theta_{ \mu_2}\}} |T_{\nu_1} HT_{\nu_2} H| \\
&\lesssim  \sum_{\{(\mu_1, \mu_2): \, \tau^\theta_{\mu_1} 
\sim \tau^\theta_{ \mu_2}\}}
\sum_{\{(\nu_1, \nu_2)\in
\tilde\tau^\theta_{\mu_1}\times \tilde\tau^\theta_{ \mu_2}\}} |A^{\theta_0}_{\nu_1}\tilde\sigma_\la  H|^2\cdot|A^{\theta_0}_{\nu_2} \tilde\sigma_\la H|^2.
    \end{aligned}
\end{equation}
Here $\tilde{\tau}^\theta_{\mu_1}$ and $\tilde{\tau}^\theta_{ \mu_2}$ are the cubes with the same centers but $11/10$ times the
side length
of $\tau^\theta_{\mu_1}$ and $\tau^\theta_{ \mu_2}$, respectively, we used the fact that the side length of $\tau^\theta_{\mu_1}$ is $\ge 2^{20}\theta_0$, so $0.1*\text{side length} \gg 2^{13}\theta_0$.

Furthermore, we  have for a 
given fixed $c_0=2^{-m_0}$, $m_0\in {\mathbb N}$, and pair of dyadic cubes $\tau^\theta_{\mu_1}$, $\tau^\theta_{ \mu_2}$
with $\tau^\theta_{\mu_1} \sim \tau^\theta_{ \mu_2}$
and $\theta=2^k\theta_0$ 
\begin{multline}\label{ll2}
\sum_{(\nu_1, \nu_2)\in
\tilde\tau^\theta_{\mu_1}\times \tilde\tau^\theta_{ \mu_2}} 
|\tilde\sigma_\la A^{\theta_0}_{\nu_1} H|^2\cdot| \tilde\sigma_\la A^{\theta_0}_{\nu_2} H|^2
\\
=\sum_{(\nu_1, \nu_2)\in
\tilde\tau^\theta_{\mu_1}\times \tilde\tau^\theta_{ \mu_2}}  \, 
\sum_{\substack{\tau^{c_0\theta}_{\mu_1'} \cap \overline{\tau}^\theta_{\mu_1} \ne \emptyset
\\ \tau^{c_0\theta}_{ \mu_2'} \cap \overline{\tau}^\theta_{ \mu_2} \ne \emptyset}}
|\tilde\sigma_\la A^{c_0\theta}_{\mu_1'}A^{\theta_0}_{\nu_1} H|^2\cdot|\tilde\sigma_\la A^{c_0\theta}_{\mu_2'}A^{\theta_0}_{\nu_2}  H|^2
+O(\la^{-N}\|H\|_2^4),
 \end{multline}
if $\overline{\tau}^\theta_{\mu_1}$ and $\overline{\tau}^\theta_{\mu_2}$ the cubes with the same centers but $12/10$ times the
side length
of $\tau^\theta_{\mu_1}$ and $\tau^\theta_{\mu_2}$, respectively, so that we have
$\text{dist}(\overline{\tau}^\theta_{\mu_1}, \overline{\tau}^\theta_{ \mu_2})\ge \theta/2$ when
$\tau^\theta_{\mu_1}  \sim \tau^\theta_{ \mu_2}$.  
This follows from the fact that for $c_0$ small enough the product of the symbol of $A_{\mu_1'}^{c_0\theta}$ and $A_{\nu_1}^{\theta_0}$ 
vanishes identically if $\tau_{\mu_1'}^{c_0\theta}\cap \overline{\tau}^\theta_{\mu_1}=\emptyset$ and $\nu_1 \in \tilde\tau^\theta_{\mu_1}$, since
$\theta=2^k\theta_0$ with $k\ge 20$.  
Also notice that we then have 
for fixed $c_0=2^{-m_0}$ small enough
\begin{equation}\label{lsep}
\text{dist}(\tau^{c_0\theta}_{\mu'_1}, \tau^{c_0\theta}_{ \mu'_2})\in [4^{-1}\theta, 4^2  \theta],
\quad \text{if } \, \, \tau^{c_0\theta}_{\mu_1'} \cap \overline{\tau}^\theta_{\mu_1} \ne \emptyset, \, \, \,
\text{and } \, \, \tau^{c_0\theta}_{ \mu_2'} \cap \overline{\tau}^\theta_{ \mu_2} \ne \emptyset.
\end{equation}
Also, for each $\mu_1$ there are $O(1)$ terms $\mu_1'$  with $\tau^{c_0\theta}_{\mu_1'} \cap \overline{\tau}^\theta_{\mu_1} \ne \emptyset$,
if $c_0$ is fixed.

If the cubes $\tau^{c_0\theta}_{\mu'_1}$, $\tau^{c_0\theta}_{\mu'_2}$ satisfy the conditions above, we have the  following bilinear type estimate which is a consequence of  Proposition 3.3 in \cite{blair2023strichartz}. 
 \begin{equation}\label{ll3}
 \iint \bigl| \tilde \sigma_\la A^{c_0\theta}_{\mu_1}H_1 \cdot \tilde \sigma_\la A^{c_0\theta}_{ \mu_2}H_2\bigr|^{\frac q2} \,  dxdt
 \lesssim 
\,\la^{1+\e} \, \bigl(2^k\theta_0\la\bigr)^{\frac{n-1}2
(q-\frac{2(n+1)}{n-1}))} \|H_1\|^{q/2}_{L^2_{t,x}} \, \|H_2\|^{q/2}_{L^2_{t,x}}.
\end{equation}
The proof of \eqref{ll3} follows from using the Hadamard parametrix to rewrite the 
$\tilde \sigma_\la A^{c_0\theta}_{\nu}$, $\nu=\mu_1, \mu_2$ operators as  oscillatory integral operators with explicit kernels and then applying Lee's oscillatory integral estimate \cite{LeeBilinear}.
The assumption that $c_0$ is small enough is crucial to ensure the separation conditions in Theorem 1.1 of \cite{LeeBilinear}, see the proof of Proposition 3.3 in  \cite{blair2023strichartz} for more details.

Note that by \eqref{org} and \eqref{ll2}, we have 
\begin{equation}\label{org1}
\begin{aligned}
&\big|{\overline\Upsilon^{\text{far}}}(H) \big|
    \le \sum_{\{k\in {\mathbb N}: \, k\ge 20 \, \, \text{and } \, 
\theta=2^k\theta_0\ll 1\}} \,  \, 
\sum_{\{(\mu_1, \mu_2): \, \tau^\theta_{\mu_1}
\sim \tau^\theta_{ \mu_2}\}}
\sum_{(\nu_1, \nu_2)\in
\tilde\tau^\theta_{\mu_1}\times \tilde\tau^\theta_{ \mu_2}}  \, 
\sum_{\substack{\tau^{c_0\theta}_{\mu_1'} \cap \overline{\tau}^\theta_{\mu_1} \ne \emptyset
\\ \tau^{c_0\theta}_{ \mu_2'} \cap \overline{\tau}^\theta_{ \mu_2} \ne \emptyset}} \\
& \qquad\qquad\qquad\qquad\qquad\qquad \Big(
|\tilde\sigma_\la A^{c_0\theta}_{\mu_1'}A^{\theta_0}_{\nu_1} H|^2\cdot|\tilde\sigma_\la A^{c_0\theta}_{\mu_2'}A^{\theta_0}_{\nu_2}  H|^2\Big)
+O(\la^{-N}\|H\|_2^4),
\end{aligned}
\end{equation}
It is not hard to see that the number of terms in the sum is $O(\la^{(2n-3)\e_0})=O(\la^{5\e_0})$, and 
by \eqref{ll3}, we have for each term in the right side of \eqref{org1}, 
\begin{equation}\label{iii1}
\begin{aligned}
\int\Big(&\int | \tilde \sigma_\la H \cdot
\tilde \sigma_\la H|^{\frac{q_e-q}2} \, 
\Big||\tilde\sigma_\la A^{c_0\theta}_{\mu_1'}A^{\theta_0}_{\nu_1} H|^2\cdot|\tilde\sigma_\la A^{c_0\theta}_{\mu_2'}A^{\theta_0}_{\nu_2}  H|^2\Big|^{q/4} \, dx\Big)^{\frac{2}{q_e}}dt \\
 &\lesssim \|\tilde S_\la f\|_{L^\infty(A_-)}^{\frac{2(q_e-q)}{q_e}} \cdot 
\|\big(\tilde\sigma_\la A^{c_0\theta}_{\mu_1'}A^{\theta_0}_{\nu_1} H\big)\cdot\big(\tilde\sigma_\la A^{c_0\theta}_{\mu_2'}A^{\theta_0}_{\nu_2}  H\big)\|_{L^{\frac{q}{q_e}}_{t}L^{\frac q2}_x}^{\frac{q}{q_e}} \\
&\lesssim T^{(\frac{q_e}{q}-\frac{2}{q})\cdot\frac{q}{q_e}} \|\tilde S_\la f\|_{L^\infty(A_-)}^{\frac{2(q_e-q)}{q_e}} \cdot 
\|\big(\tilde\sigma_\la A^{c_0\theta}_{\mu_1'}A^{\theta_0}_{\nu_1} H\big)\cdot\big(\tilde\sigma_\la A^{c_0\theta}_{\mu_2'}A^{\theta_0}_{\nu_2}  H\big)\|_{L^{\frac{q}{2}}_{t}L^{\frac q2}_x}^{\frac{q}{q_e}},
\\
&\lesssim  T^{(\frac{q_e}{q}-\frac{2}{q})\cdot\frac{q}{q_e}} \la ^{(\frac{n-1}4+\e_1)(\frac{2(q_e-q)}{q_e})}
\cdot \Big(\la^{1+\e} \, \bigl(2^k\la^{1-\e_0}\bigr)^{\frac{n-1}2
(q-\frac{2(n+1)}{n-1})} \Big)^{\frac{2}{q_e}}\|H\|^{\frac{2q}{q_e}}_{L^2_{x}}.
\end{aligned}
\end{equation}
As in \eqref{far}, if we take $\e, \e_0$ and $\e_1$ to be small enough, e.g., $\e=\e_0=\e_1=\frac1{100}$  it is straightforward to check that the right side of \eqref{iii1} is 
$O(\la^{1-6\e_0}\|H\|_{L^2_{t,x}}^{2}),
$ using the fact that 
$\|H\|_{L^2_{t,x}}^{2}$ dominates 
$\|H\|_{L^2_{t,x}}^{\frac{2q}{q_e}}$ since we are 
 assuming $f$ is $L^2$ normalized.
This implies that 
\begin{equation}\label{bn2''}
\begin{aligned}
B=&\int\Big(\int | \tilde \sigma_\la H \cdot
\tilde \sigma_\la H|^{\frac{q_e-q}2} \, 
|{\overline\Upsilon^{\text{far}}}(H)|^{q/4} \, dx\Big)^{\frac{2}{q_e}}dt \\
\lesssim & \la^{1-6\e_0+5\e_0}\|H\|_{L^2_{t,x}}^{2},
\end{aligned}
\end{equation}
which finishes the proof of Proposition~\ref{locprop} for $n-1=3$.

\bigskip

\noindent{\bf Remarks about $n-1=2$}.
The arguments for $n-1=2$ is similar to the case $n-1=3$. Recall that when $n-1=3$, we are essentially doing another round of bilinear estimates  within the diagonal terms in the Whitney decomposition due to the fact that $q_e=\frac{2(n-1)}{n-3}=6>4$. When $n-1=2$, $q$ can be arbitrary large, if $q\in [2^{k+1}, 2^{k+2}]$ for some $k\in \mathbb{N}^+$, then one can repeat the above arguments  $k$ times,  the resulting diagonal term will involve a product of $2^{k+1}$ terms of type $\tilde \sigma_\la \Atn H$, and it satisfies the natural analog of Lemma~\ref{blemma} since $\frac{q}{2^{k+1}}\in [1,2]$. And there are off-diagonal terms each time we run the Whitney decomposition, those terms can be treated as in \eqref{iii1} by using bilinear estimates. The main difference is, as $q\rightarrow \infty$, unlike \eqref{far} and \eqref{iii1}, we have to take $\e_0$ and $\e_1$ to be small enough depending on $q$, instead of some fixed small constant.

\bigskip

\bibliography{refs}

\providecommand{\MR}[1]{}
\begin{thebibliography}{10}

\bibitem{Berard}
P.~H. B\'erard.
\newblock On the wave equation on a compact {R}iemannian manifold without
  conjugate points.
\newblock {\em Math. Z.}, 155(3):249--276, 1977.

\bibitem{BHSsp}
M.~D. Blair, X.~Huang, and C.~D. Sogge.
\newblock Improved spectral projection estimates.
\newblock {\em to appear in J. Eur. Math. Soc. (JEMS)}.

\bibitem{blair2023strichartz}
M.~D. Blair, X.~Huang, and C.~D. Sogge.
\newblock Strichartz estimates for the $\text{S}$chr\" odinger equation on
  negatively curved compact manifolds.
\newblock {\em arXiv:2304.05247}, 2023.

\bibitem{BlairSoggeRefined}
M.~D. Blair and C.~D. Sogge.
\newblock {Refined and microlocal {K}akeya-{N}ikodym bounds for eigenfunctions
  in two dimensions}.
\newblock {\em Anal. PDE}, 8(3):747--764, 2014.

\bibitem{blair2015refined}
M.~D. Blair and C.~D. Sogge.
\newblock Refined and microlocal {K}akeya-{N}ikodym bounds of eigenfunctions in
  higher dimensions.
\newblock {\em Comm. Math. Phys.}, 356(2):501--533, 2017.

\bibitem{BSTop}
M.~D. Blair and C.~D. Sogge.
\newblock Concerning {T}oponogov's theorem and logarithmic improvement of
  estimates of eigenfunctions.
\newblock {\em J. Differential Geom.}, 109(2):189--221, 2018.

\bibitem{SBLog}
M.~D. Blair and C.~D. Sogge.
\newblock Logarithmic improvements in {$L^p$} bounds for eigenfunctions at the
  critical exponent in the presence of nonpositive curvature.
\newblock {\em Invent. Math.}, 217(2):703--748, 2019.

\bibitem{BourgainBesicovitch}
J.~Bourgain.
\newblock {Besicovitch type maximal operators and applications to {F}ourier
  analysis}.
\newblock {\em Geom. Funct. Anal.}, 1(2):147--187, 1991.

\bibitem{bourgain1993fourier}
J.~Bourgain.
\newblock Fourier transform restriction phenomena for certain lattice subsets
  and applications to nonlinear evolution equations: Part $\text{I}$:
  Schr\"odinger equations.
\newblock {\em Geometric \& Functional Analysis GAFA}, 3(3):107--156, 1993.

\bibitem{BoDe}
J.~Bourgain and C.~Demeter.
\newblock The proof of the {$l^2$} decoupling conjecture.
\newblock {\em Ann. of Math.}, 182:351--389, 2015.

\bibitem{bgtmanifold}
N.~Burq, P.~G{\'e}rard, and N.~Tzvetkov.
\newblock {Strichartz inequalities and the nonlinear {S}chr{\"o}dinger equation
  on compact manifolds}.
\newblock {\em Amer. J. Math.}, 126(3):569--605, 2004.

\bibitem{DGG}
Y.~Deng, P.~Germain, and L.~Guth.
\newblock Strichartz estimates for the {S}chr\"{o}dinger equation on irrational
  tori.
\newblock {\em J. Funct. Anal.}, 273(9):2846--2869, 2017.

\bibitem{DGGM}
Y.~Deng, P.~Germain, L.~Guth, and S.~L. Rydin~Myerson.
\newblock Strichartz estimates for the {S}chr\"{o}dinger equation on
  non-rectangular two-dimensional tori.
\newblock {\em Amer. J. Math.}, 144(3):701--745, 2022.

\bibitem{HassellTacy}
A.~Hassell and M.~Tacy.
\newblock {Improvement of eigenfunction estimates on manifolds of nonpositive
  curvature}.
\newblock {\em Forum Mathematicum}, 27(3):1435--1451, 2015.

\bibitem{herr2023strichartz}
S.~Herr and B.~Kwak.
\newblock Strichartz estimates and global well-posedness of the cubic
  $\text{NLS}$ on $\mathbb{T}^2$.
\newblock {\em arXiv:2309.14275}, 2023.

\bibitem{HSSchro}
X.~Huang and C.~D. Sogge.
\newblock Quasimode and {S}trichartz estimates for time-dependent
  {S}chr\"odinger equations with singular potentials.
\newblock {\em Math Research Letters}, 29:727--762, 2022.

\bibitem{KT}
M.~Keel and T.~Tao.
\newblock Endpoint {S}trichartz estimates.
\newblock {\em Amer. J. Math.}, 120(5):955--980, 1998.

\bibitem{LeeBilinear}
S.~Lee.
\newblock {Linear and bilinear estimates for oscillatory integral operators
  related to restriction to hypersurfaces}.
\newblock {\em J. Funct. Anal.}, 241(1):56--98, 2006.

\bibitem{sogge88}
C.~D. Sogge.
\newblock {Concerning the {$L^p$} norm of spectral clusters for second-order
  elliptic operators on compact manifolds}.
\newblock {\em J. Funct. Anal.}, 77(1):123--138, 1988.

\bibitem{SFIO2}
C.~D. Sogge.
\newblock {\em Fourier integrals in classical analysis}, volume 210 of {\em
  Cambridge Tracts in Mathematics}.
\newblock Cambridge University Press, Cambridge, second edition, 2017.

\bibitem{sogge2015improved}
C.~D. Sogge.
\newblock {Improved critical eigenfunction estimates on manifolds of
  nonpositive curvature}.
\newblock {\em Math. Res. Lett.}, 24:549--570, 2017.

\bibitem{SoggeZelditchL4}
C.~D. Sogge and S.~Zelditch.
\newblock {On eigenfunction restriction estimates and {$L^4$}-bounds for
  compact surfaces with nonpositive curvature}.
\newblock In {\em {Advances in analysis: the legacy of {E}lias {M}. {S}tein}},
  volume~50 of {\em {Princeton Math. Ser.}}, pages 447--461. Princeton Univ.
  Press, Princeton, NJ, 2014.

\bibitem{TaoVargasVega}
T.~Tao, A.~Vargas, and L.~Vega.
\newblock A bilinear approach to the restriction and {K}akeya conjectures.
\newblock {\em J. Amer. Math. Soc.}, 11(4):967--1000, 1998.

\bibitem{TaylorPDO}
M.~E. Taylor.
\newblock {\em Pseudodifferential operators}.
\newblock Princeton Mathematical Series, No. 34. Princeton University Press,
  Princeton, N.J., 1981.

\end{thebibliography}
\bibliographystyle{abbrv}

%

\end{document}